\documentclass[reqno,a4paper]{amsart}
\usepackage{amssymb}
\usepackage{amsmath}

\begin{document} 
\newtheorem{prop}{Proposition}[section]
\newtheorem{Def}{Definition}[section] \newtheorem{theorem}{Theorem}[section]
\newtheorem{lemma}{Lemma}[section] \newtheorem{Cor}{Corollary}[section]

\title[Klein - Gordon - Schr\"odinger system]{\bf Some new well-posedness 
results for the Klein - Gordon - Schr\"odinger system}
\author[Hartmut Pecher]{
{\bf Hartmut Pecher}\\
Fachbereich Mathematik und Naturwissenschaften\\
Bergische Universit\"at Wuppertal\\
Gau{\ss}str.  20\\
42097 Wuppertal\\
Germany}
\email{\tt pecher@math.uni-wuppertal.de}
\date{}

\begin{abstract} 
We consider the Cauchy problem for the 2D and 3D Klein-Gordon-Schr\"odinger 
system. In 2D we show local well-posedness for Schr\"odin-\\ger data in $H^s$ 
and wave data in $H^{\sigma} \times H^{\sigma -1}$ for $s=-1/4 \, +$ and 
$\sigma 
= -1/2$, whereas ill-posedness holds for $s<- 1/4$ or $\sigma <-1/2$, and global 
well-posedness 
for $s\ge 0$ and $s-\frac{1}{2} \le \sigma < s+ \frac{3}{2}$. In 3D we show 
global well-posedness for $s \ge 0$ , $ s - \frac{1}{2} < \sigma \le s+1$.
Fundamental for our results are the studies by Bejenaru, Herr, Holmer and 
Tataru \cite{BHHT}, and Bejenaru and Herr \cite{BH} for the Zakharov system, 
and also the global well-posedness results for the Zakharov and 
Klein-Gordon-Schr\"odinger system by Colliander, Holmer and Tzirakis 
\cite{CHT}.
\end{abstract}
\maketitle

\renewcommand{\thefootnote}{\fnsymbol{footnote}}
\footnotetext{\hspace{-1.8em}{\it 2000 Mathematics Subject Classification:} 
35Q55, 35L70 }

\normalsize 
\setcounter{section}{0}
\section{Introduction and main results}

We consider the Cauchy problem for the Klein - Gordon - 
Schr\"odinger system with Yukawa coupling
\begin{eqnarray}
\label{0.1}
i\partial_t u + \Delta u & = & nu
\\
\label{0.2}
\partial_t^2 n + (1-\Delta) n & = & |u|^2
\end{eqnarray}
with initial data
\begin{equation}
u(0)  =  u_0 \,  , \, n(0)  =  n_0 \, , \, \partial_t 
n(0) = n_1 \, ,
\label{0.3}
\end{equation}
where $u$ is a complex-valued and $n$ a real-valued function defined for $(x,t) 
\in {\mathbb R}^D \times [0,T]$ , $D=2$ or $D=3$ .
This is a classical model which describes a system of scalar nucleons 
interacting with neutral scalar mesons. The nucleons are described by the 
complex scalar field $u$ and the mesons by the real scalar field $n$. The mass 
of the meson is normalized to be 1.

Our results do not use the energy conservation 
law but only charge conservation $ \|u(t)\|_{L^2({\mathbb R}^D)} \equiv const $ 
(for 
the global existence result), so they are equally true if one replaces $nu$ and 
$|u|^2$ by $-nu$ and/or $-|u|^2$ , respectively.

We are interested in local and global solutions for data
$$ u_0 \in H^s({\mathbb R}^D) \, , \, n_0 \in H^{\sigma}({\mathbb R}^D) \, , \, 
n_1 \in 
H^{\sigma -1}({\mathbb R}^D) \, .$$

In the case $D=3$ local well-posedness in Bourgain type spaces was proven by 
the 
author \cite{P} under the assumptions
$$ s > - \frac{1}{4} \ , \, \sigma > - \frac{1}{2} \, , \, \sigma -2s < 
\frac{3}{2} \, , \, \sigma - 2 <s < \sigma +1 \, . $$
Moreover it was shown that up to the endpoints these conditions are sharp.

Global well-posedness in $D=3$ in spaces of Strichartz type was shown by 
Colliander, 
Holmer and Tzirakis \cite{CHT} in the case $s=\sigma \ge 0$. This is also true 
in $D=2$ by similar arguments. Unconditional uniqueness in the natural solution 
spaces in this case also holds \cite{P}.

In the case $D=2$ we are now able to show that local well-posedness in Bourgain 
type spaces holds under the same assumptions as in $D=3$ including the case 
$\sigma = 
-\frac{1}{2}$ (Theorem \ref{Theorem 2.1}). The ill-posedness statement also 
carries over to the case $D=2$ (Theorem \ref{Theorem 2.2}).
 
We also show global well-posedness in $D=2$ for $u_0 \in L^2$ , $ n_0 \in 
H^{\sigma}$ , $ n_1 \in H^{\sigma -1}$, if $-\frac{1}{2} \le \sigma < 
\frac{3}{2}$, and more generally for $u_0 \in H^s$ , $n_0 \in H^{\sigma}$ , $ 
n_1 \in H^{\sigma -1}$, if $s \ge 0$ , $ s-\frac{1}{2} \le \sigma < s + 
\frac{3}{2}$ (Theorem \ref{Theorem 3.2}).

In the cases $D=2$ and $D=3$ we show global well-posedness in the case $0 \le s 
\le \sigma \le s+1$ in spaces of Strichartz type (Theorem \ref{Theorem 5.1}), 
and in the case $ s \ge 0$ , $ 
s-\frac{1}{2} < \sigma < s$ in spaces of Bourgain type (Theorem \ref{Theorem 
1.6'}).

In any of these cases unconditional uniqueness holds if $s,\sigma \ge 0$ in 
$D=2$ and $D=3$ (cf. Theorem \ref{Theorem 3.2} and Theorem \ref{Theorem 5.2}, 
respectively).

The results in this paper are based on the (3+1)-dimensional estimates by 
Bejenaru and Herr \cite{BH} which they recently used to show a sharp 
well-posedness result for the Zakharov system. We also use the corresponding 
sharp (2+1)-dimensional local well-posedness results for the Zakharov system by 
Bejenaru, Herr, Holmer and Tataru \cite{BHHT}.

Concerning the closely related wave Schr\"odinger system local well-posedness 
in 
$D=3$ 
was shown for $ s > - \frac{1}{4} $ and $ \sigma > - \frac{1}{38} $ and also 
global well-posedness for certain $ s,\sigma < 0$ by T. Akahori 
\cite{A}.

We use the standard Bourgain spaces $X^{m,b}$ for the Schr\"odinger equation, 
which are defined as the completion of ${\mathcal S}({\mathbb R}^D \times 
{\mathbb R})$ 
with respect to
$$ \|f\|_{X^{m,b}} := \| \langle \xi \rangle ^m \langle \tau + |\xi|^2 \rangle 
^b \widehat{f}(\xi,\tau)\|_{L^2_{\xi\tau}} \, . $$
Similarly $X^{m,b}_{\pm}$ for the equation $i\partial_t n_{\pm} \mp 
A^{1/2}n_{\pm} = 0$ is the completion of ${\mathcal S}({\mathbb R}^D \times 
{\mathbb R})$ 
with respect to
$$ \|f\|_{X^{m,b}_{\pm}} := \| \langle \xi \rangle ^m \langle \tau \pm |\xi| 
\rangle ^b \widehat{f}(\xi,\tau)\|_{L^2_{\xi\tau}} \, . $$
For a given time interval $I$ we define $ \|f\|_{X^{m,b}(I)} := 
\inf_{\tilde{f}_{|I} = f} \|\tilde{f}\|_{X^{m,b}} $ and similarly $ 
\|f\|_{X^{m,b}_{\pm}(I)}$ . We often skip $I$ from the notation.

In the following we mean by a solution of a system of differential equation
always a solution of the corresponding system of integral equations.

Before formulating the main results of our paper we recall that the KGS system 
can be transformed into a first order (in t) system as follows: if
$$ (u,n,\partial_t n) \in C^0([0,T],H^s) \times  C^0([0,T],H^{\sigma}) \times  
C^0([0,T],H^{\sigma -1}) $$ is a solution of 
(\ref{0.1}),(\ref{0.2}),(\ref{0.3}) 
with data $(u_0,n_0,n_1) \in H^s \times H^{\sigma} \times H^{\sigma -1}$ ,then 
defining  $A:= - 
\Delta + 1$ and 
$$ n_{\pm} := n \pm iA^{-\frac{1}{2}} \partial_t n  $$ 
and 
$$ n_{\pm 0} := n_0 \pm iA^{-\frac{1}{2}}n_1 \in H^{\sigma} \, ,$$
 we get that
$$(u,n_+,n_-) \in C^0([0,T],H^s) \times  C^0([0,T],H^{\sigma}) \times  
C^0([0,T],H^{\sigma})$$
is a solution of the following problem:
\begin{eqnarray}
\label{0.1'}
i \partial_t u + \Delta u & = & \frac{1}{2}(n_+ - n_-)u \\
\label{0.2'}
i \partial_t n_{\pm} \mp A^{1/2} n_{\pm} & = & \pm A^{-1/2}(|u|^2) \\
\label{0.3'}
u(0) = u_0 & , & n_{\pm}(0) = n_{\pm 0} := n_0 \pm i A^{-1/2}n_1 \, . 
\end{eqnarray}
The corresponding system of integral equations reads as follows: 
\begin{eqnarray}
\label{8}
u(t) & = & e^{it\Delta} u_0 + \frac{1}{2} \int_0^t e^{i(t-\tau)\Delta} 
(n_+(\tau)+n_-(\tau))u(\tau) d\tau \\
\label{9}
n_{\pm}(t) & = & e^{\mp itA^{1/2}}n_{\pm 0} \pm i \int_0^t e^{\mp 
i(t-\tau)A^{1/2}} 
A^{-1/2}(|u(\tau)|^2) d\tau \, .
\end{eqnarray}

Conversely, if 
$$(u,n_+,n_-) \in X^{s,b}[0,T] \times X_+^{\sigma,b}[0,T]  \times 
X_-^{\sigma,b}[0,T]$$
is a solution of (\ref{0.1'}),(\ref{0.2'}) with data $u(0)=u_0 \in H^s$ and 
$n_{\pm}(0) = n_{\pm 0} \in H^{\sigma}$ , then we define
$n := \frac{1}{2}(n_+ + n_-)$ , $ 2i A^{-\frac{1}{2}}\partial_t n := n_+ - n_-$ 
and conclude that
$$(u,n,\partial_t n) \in X^{s,b}[0,T] \times (X_+^{\sigma,b}[0,T] + 
X_-^{\sigma,b}[0,T]) \times (X_+^{\sigma -1,b}[0,T] + X_-^{\sigma -1,b}[0,T]) 
$$
is a solution of (\ref{0.1}),(\ref{0.2}) with data 
$ u(0)=u_0 \in H^s$ and
 $$ n(0)=n_0= \frac{1}{2}(n_+(0)+n_-(0)) \in H^{\sigma} \, , \, \partial_t n(0) 
= \frac{1}{2i} A^{\frac{1}{2}}(n_+(0)-n_-(0)) \in H^{\sigma -1} \, . $$
If $(u,n_+,n_-) \in C^0([0,T],H^s) \times  C^0([0,T],H^{\sigma}) \times  
C^0([0,T],H^{\sigma}) $ , then we also have $(u,n,\partial_t n) \in 
C^0([0,T],H^s) \times  C^0([0,T],H^{\sigma}) \times  C^0([0,T],H^{\sigma -1})$ 
.\\[1em]
Our local well-posedness result in 2D reads as follow:
\begin{theorem}
\label{Theorem 2.1}
The Klein - Gordon - Schr\"odinger system 
(\ref{0.1}),(\ref{0.2}),(\ref{0.3}) in 2D is locally well-posed for data 
$$ u_0 \in H^s({\mathbb R}^2) \, , \, n_0 \in H^{\sigma}({\mathbb R}^2) \, , \,  
n_1 
\in H^{\sigma -1}({\mathbb R}^2)$$
under the assumptions
$$ s > - \frac{1}{4} \, , \, \sigma \ge - \frac{1}{2} \, , \, \sigma -2s < 
\frac{3}{2} \, , \, \sigma -2 < s < \sigma +1 \, . $$
More precisely, there exists $ T > 0 $ , 
$T=T(\|u_0\|_{H^s},\|n_0\|_{H^{\sigma}},\|n_1\|_{H^{\sigma -1}})$ and a unique 
solution
$$ u \in X^{s,\frac{1}{2}+}[0,T] \, , $$
$$ n \in X_+^{\sigma,\frac{1}{2}+}[0,T] + 
X_-^{\sigma,\frac{1}{2}+}[0,T] \, , \, \partial_t n  \in X_+^{\sigma 
-1,\frac{1}{2}+}[0,T] + X_-^{\sigma -1,\frac{1}{2}+}[0,T] \, . $$
This solution has the property
$$ u \in C^0([0,T],H^s({\mathbb R}^2)) \, , \,  n \in 
C^0([0,T],H^{\sigma}({\mathbb 
R}^2)) \, , \, \partial_t n \in 
C^0([0,T],H^{\sigma -1}({\mathbb R}^2)) \, . $$
Under the additional assumption $s,\sigma \ge 0$ we also have (unconditional) 
uniqueness in these latter spaces.
\end{theorem}

These conditions are sharp up to the endpoints. We namely have the following 
result, which can be proven exactly as in the case $D=3$.
\begin{theorem}
\label{Theorem 2.2}
Let $ u_0 \in H^s({\mathbb R}^2)$ ,  $ n_0 \in H^{\sigma}({\mathbb R}^2)$ , $ 
n_1 \in 
H^{\sigma -1}({\mathbb R}^2)$ . Then the flow map $ (u_0,n_0,n_1) \mapsto 
(u(t),n(t),\partial_t n(t)) $ , $ t \in [0,T] $ , does not belong to $C^2$ for 
any $ T > 0 $ , provided $ \sigma -2s-\frac{3}{2} > 0 $ or $ s < - \frac{1}{4} 
$ 
or $ 
\sigma < - \frac{1}{2} $ .
\end{theorem}

The global well-posedness result for $D=2$ in the case of $L^2$-Schr\"odinger 
data is the following
\begin{theorem}
\label{Theorem 3.2}
The Klein - Gordon - Schr\"odinger system 
(\ref{0.1}),(\ref{0.2}),(\ref{0.3}) in 2D is globally well-posed for data 
$$ u_0 \in H^s({\mathbb R}^2) \, , \, n_0 \in H^{\sigma}({\mathbb R}^2) \, , \,  
n_1 
\in H^{\sigma -1}({\mathbb R}^2)$$
under the assumptions
$$ s \ge 0 \, , \, s - \frac{1}{2} \le \sigma < s+\frac{3}{2} \, ,$$ 
 i.e. for any 
$ T > 0 $ there exists a unique 
solution
$$ u \in X^{s,\frac{1}{2}+}[0,T] \, , $$
$$ n \in X_+^{\sigma,\frac{1}{2}+}[0,T] + 
X_-^{\sigma,\frac{1}{2}+}[0,T] \, , \, \partial_t n  \in X_+^{\sigma 
-1,\frac{1}{2}+}[0,T] + X_-^{\sigma -1,\frac{1}{2}+}[0,T] \, . $$
This solution has the property
$$ u \in C^0([0,T],H^s({\mathbb R}^2)) \, , \,  n \in 
C^0([0,T],H^{\sigma}({\mathbb 
R}^2)) \, , \, \partial_t n \in 
C^0([0,T],H^{\sigma -1}({\mathbb R}^2)) \, . $$
Under the additional assumption $\sigma \ge 0$ we also have (unconditional) 
uniqueness in these latter spaces,
especially there exists a unique global classical solution for smooth data.
\end{theorem}

A global well-posedness result in $2D$ and also in $3D$ in the range $0 \le s 
\le \sigma \le s+1$ can be proven without using Bourgain type spaces but only 
Strichartz type estimates.

\begin{theorem}
\label{Theorem 5.1} Let the space dimension $D$ be 2 or 3.
Assume $0 \le s \le \sigma \le s+1$ and $u_0 \in H^s$, $n_0 \in H^{\sigma}$ , $ 
n_1 \in H^{\sigma -1}$. Then the Klein - Gordon - Schr\"odinger system 
(\ref{0.1}),(\ref{0.2}),(\ref{0.3}) is globally well-posed, i.e. for any 
$T>0$ there exists a unique solution 
$$u \in C^0([0,T],H^s) \cap \bigcap_{2\le r < \infty , 2 \le q \le \infty} 
L^q((0,T),H^{s,r}) \, , $$
$$ n \in C^0([0,T],H^{\sigma}) \, , \, \partial_t n \in C^0([0,T];H^{\sigma 
-1}) 
\, , $$ 
where $ \frac{2}{q} + \frac{D}{r} = \frac{D}{2}$ .
\end{theorem}

For negative $\sigma$ we have to use Bourgain type spaces again.
\begin{theorem}
\label{Theorem 1.6'}
The Klein - Gordon - Schr\"odinger system 
(\ref{0.1}),(\ref{0.2}),(\ref{0.3}) in 3D is globally well-posed for data 
$$ u_0 \in H^s({\mathbb R}^3) \, , \, n_0 \in H^{\sigma}({\mathbb R}^3) \, , \,  
n_1 
\in H^{\sigma -1}({\mathbb R}^3)$$
under the assumptions
$$ s \ge 0 \, , \, s - \frac{1}{2} < \sigma < s \, ,$$ 
 i.e. for any 
$ T > 0 $ there exists a unique 
solution
$$ u \in X^{s,\frac{1}{2}+}[0,T] \, , $$
$$ n \in X_+^{\sigma,\frac{1}{2}+}[0,T] + 
X_-^{\sigma,\frac{1}{2}+}[0,T] \, , \, \partial_t n  \in X_+^{\sigma 
-1,\frac{1}{2}+}[0,T] + X_-^{\sigma -1,\frac{1}{2}+}[0,T] \, . $$
This solution has the property
$$ u \in C^0([0,T],H^s({\mathbb R}^3)) \, , \,  n \in 
C^0([0,T],H^{\sigma}({\mathbb 
R}^3)) \, , \, \partial_t n \in 
C^0([0,T],H^{\sigma -1}({\mathbb R}^3)) \, . $$
\end{theorem}

\noindent {\bf Remark:} It would be desirable to have a similar result in the 
case $s=0$ , $1 < \sigma < \frac{3}{2}$ as in the case $D=2$, but our estimates 
given in spaces of Bourgain type seem to be not quite strong enough to prove 
this.

Combining this with the unconditional uniqueness result \cite{P} in the case 
$s=\sigma =0$ we also get

\begin{theorem}
\label{Theorem 5.2}
Assume $s,\sigma \ge 0$ and $s-\frac{1}{2} < \sigma \le s+1$ and $u_0 \in H^s$, 
$n_0 \in H^{\sigma}$ , $ 
n_1 \in H^{\sigma -1}$.
Then the Klein - Gordon - 
Schr\"odinger system 
(\ref{0.1}),(\ref{0.2}),(\ref{0.3}) in 3D has a unique global solution 
$$u \in C^0([0,T],H^s)  \, , \,
n \in C^0([0,T],H^{\sigma}) \, , \, \partial_t n \in C^0([0,T];H^{\sigma -1}) 
\, 
. $$ 
\end{theorem}

Concerning the standard facts for the linear Cauchy problem (which are 
independent of the specific phase function) in spaces of Bourgain type we refer 
to \cite[Section 2]{GTV} or \cite{G}. We also use the following well-known fact 
\cite[Lemma 1.10]{G}, which we prove for the sake of completeness.

\begin{lemma}
\label{Lemma}
If $ s \in \mathbb R$ , $T\le 1$ , $0 \le b' < b < \frac{1}{2}$ or $0 \ge b > 
b' 
> - \frac{1}{2}$ , the following estimate holds:
$$
\|u\|_{X^{s,b'}[0,T]} \lesssim T^{b-b'} \|u\|_{X^{s,b}[0,T]}\, ,
$$ 

\end{lemma}
\begin{proof}
Let $\psi$ be a smooth time-cutoff function , $\psi_T(t) = \psi(\frac{t}{T})$. 
and $0 \le b' < b < \frac{1}{2}$. By the well-known Sobolev multiplication law 
in 1D we get for $0 \le s < s_1,s_2$ and $s \le s_1 + s_2 -\frac{1}{2}$:
$$ \|fg\|_{H^s} \lesssim \|f\|_{H^{s_1}} \|g\|_{H^{s_2}} \, . $$
Thus
$$ \|\psi_T u\|_{H^{b'}_t} \lesssim \|\psi_T\|_{H^{\frac{1}{2}-(b-b')}} 
\|u\|_{H^b_t} \lesssim T^{b-b'} \|u\|_{H^{b}_t} \, , $$
so that
$$ \|\psi_T u\|_{X^{s,b'}} \lesssim \|e^{-it\Delta} \psi_T u \|_{H^{b'}_t 
H^s_x} 
\lesssim T^{b-b'} \|e^{-it\Delta} u\|_{H^b_t H^s_x} = T^{b-b'} \|u\|_{X^{s,b}} 
\, , $$
which is enough to prove the claimed estimate.
The case $0 \ge b > b' > - \frac{1}{2}$ follows by duality.
\end{proof}
This obviously also holds for the spaces $X_{\pm}^{s,b}$.

The wellknown Strichartz estimates are collected in 
\begin{prop}
(Schr\"odinger equation) \\
Let $2 \le q,\tilde{q} \le \infty$ , $2\le r,\tilde{r} \le \infty$ (excluding 
$r,\tilde{r} =\infty$ in the case $D=2$), $\frac{2}{q} + \frac{D}{r} = 
\frac{D}{2},$  $\frac{2}{\tilde{q}} + \frac{D}{\tilde{r}} = \frac{D}{2}$ , 
$\frac{1}{\tilde{r}} + \frac{1}{\tilde{r}'} = 1 = \frac{1}{\tilde{q}} + 
\frac{1}{\tilde{q}'} $.
Then for any interval $I=(0,T)$:
\begin{eqnarray}
\label{H'}
\|e^{\pm it\Delta} u_0\|_{L_t^q(I,L^r_x)} & \lesssim & \|u_0\|_{L^2_x} \, ,\\
\label{I'}
\| \int_0^t e^{\pm i(t-s)\Delta} u(s) ds \|_{L^q_t(I, L^r_x)} & \lesssim & 
\|u\|_{L^{\tilde{q}'}_t(I, L^{\tilde{r}'}_x)} \, .
 \end{eqnarray}
(Klein-Gordon equation) for $D \ge 2$ :\\
Let $2 \le q,\tilde{q} \le \infty$ , $2\le r,\tilde{r} < \infty$ , $\frac{2}{q} 
+ \frac{D-1}{r} = 
\frac{D-1}{2},$  $\frac{2}{\tilde{q}} + \frac{D-1}{\tilde{r}} = \frac{D-1}{2}$ 
, 
$\frac{1}{\tilde{r}} + \frac{1}{\tilde{r}'} = 1 = \frac{1}{\tilde{q}} + 
\frac{1}{\tilde{q}'} $, $\mu = D(\frac{1}{2} - \frac{1}{r})-\frac{1}{q}$ , $ \mu 
= 1 + \rho 
-D(\frac{1}{2}-\frac{1}{\tilde{r}}) + \frac{1}{\tilde{q}} $.
Then for any interval $I=(0,T)$:
\begin{eqnarray}
\label{H}
\|e^{\pm it(1-\Delta)^{\frac{1}{2}}} u_0\|_{L^q_t (I,L^r_x)} & \lesssim & 
\|u_0\|_{H^{\mu}_x} \, ,\\
\label{I''}
\| \int_0^t e^{\pm i(t-s)(1-\Delta)^{\frac{1}{2}}} (1-\Delta)^{-\frac{1}{2}} 
u(s) ds \|_{L^q_t(I, L^r_x)} & \lesssim & 
\|u\|_{L^{\tilde{q}'}_t(I, H^{\rho,\tilde{r}'}_x)} \, ,
\end{eqnarray}
where the implicit constants are independent of $I$.
\end{prop}
In the Klein-Gordon case the proof of (\ref{H}) can be found in \cite{P1}. The 
proof of (\ref{I''}) then follows by the well-known $TT^*$ - method, as 
described in \cite{GV}, in combination with the Christ-Kiselev lemma \cite{CK}. 
In the Schr\"odinger case (\ref{I'}) follows in the same way from the standard 
estimate (\ref{H'}). 

We use the following notation. The Fourier transform is denoted by  $ \, 
\widehat{}$ or ${\mathcal F}$ , where it should be clear from the context, 
whether it is taken with 
respect to the space and time variables simultaneously or only with respect to 
the space variables. $A \lesssim B$ and $A \gtrsim B$ is shorthand for $A \le c 
B$ and $A \ge c B$, respectively, with a positive constant $c$, and $A \sim B$ 
means that $A \lesssim B$ and $A \gtrsim B$. For real numbers $a$ we denote by 
$a+$ and $a-$ a number 
sufficiently close to $a$, but
larger and smaller than $a$, respectively. \\

\section{Local well-posedness for $D=2$.}
We now formulate und prove the decisive bilinear estimates. We follow closely 
the arguments and notation from \cite{BH}.
\begin{prop}
\label{Prop. 2.1}
The following estimate holds
$$ \|un\|_{X^{0,-\frac{5}{12}-}} \lesssim \|u\|_{X^{0,\frac{5}{12}+}} 
\|n\|_{X^{-\frac{1}{2},\frac{5}{12}+}_{\pm}} \, . $$
\end{prop}

Because we are going to use dyadic decompositions of $\widehat{u}$ and 
$\widehat{v}$ we take the notation from \cite{BH} and start by choosing a 
function $\psi \in C^{\infty}_0((-2,2))$ , which is even and  nonnegative with 
$\psi(r) = 1 $ for $|r|\le 1$. Defining $\psi_N(r) = \psi(\frac{r}{N}) - 
\psi(\frac{2r}{N}) $ for dyadic numbers $N=2^n \ge 2$ and $\psi_1 = \psi$ we 
have $ 1 = \sum_{N \ge 1} \psi_N $ . Thus $supp \, \psi_1 \subset[-2,2]$ and 
$supp \,
\psi_N \subset [-2N,-N/2] \cup [N/2,2N] $ for $N \ge 2$. For $f:{\mathbb R}^2 
\to 
{\mathbb C}$ we define the dyadic frequency localization operators $P_N$ by
$$ {\mathcal F}_x(P_N f)(\xi) = \psi_N(|\xi|) {\mathcal F}_x f(\xi) \, . $$
For $u:{\mathbb R}^2 \times {\mathbb R} \to {\mathbb C}$ we define the 
modulation 
localization operators
\begin{eqnarray*}
{\mathcal F}(S_Lu)(\xi,\tau) & = & \psi_L(\tau + |\xi|^2) {\mathcal 
F}u(\xi,\tau) \\
{\mathcal F}(W_L^{\pm} u)(\xi,\tau) & = & \psi_L(\tau \pm |\xi|) {\mathcal 
F}u(\xi,\tau)
\end{eqnarray*}
in the Schr\"odinger case and the wave case.

We also define an equidistant partition of unity in ${\mathbb R}$,
$$ 1 = \sum_{j \in {\mathbb
Z}} \beta_j \, , \, \beta_j(s) = \psi(s-j) (\sum_{k 
\in {\mathbb Z}} \psi(s-k))^{-1} \, . $$
Finally, for $A \in {\mathbb N}$ we define an equidistant partition of unity on 
the unit circle
$$ 1 = \sum_{j=0}^{A-1} \beta_j^A \, , \, \beta_j^A(\theta) = 
\beta_j(\frac{A\theta}{\pi}) + \beta_{j-A}(\frac{A\theta}{\pi}) \, .$$
Then $ supp (\beta_j^A) \subset \Theta_j^A$ , where
$$ \Theta_j^A := [\frac{\pi}{A}(j-2),\frac{\pi}{A}(j+2)] \cup [-\pi + 
\frac{\pi}{A}(j-2),-\pi + \frac{\pi}{A}(j-2)] \, . $$
Now we introduce the angular frequency localization operators $Q_j^A$ by
$$ {\mathcal F}(Q_j^A f)(\xi) = \beta_j^A(\theta) {\mathcal F} f(\xi) \, , $$
where $\xi = |\xi| (\cos \theta,\sin \theta)$ . For $A \in {\mathbb N}$ we can 
now decompose $u: {\mathbb R}^2 \times {\mathbb R} \to {\mathbb C} $ as
$$ u = \sum_{j=0}^{A-1} Q_j^A u \, . $$

\begin{proof} [Proof of Proposition \ref{Prop. 2.1}]  Defining
\begin{equation}
\label{I}
I(f,g_1,g_2) = \int_{*} f(\xi_3,\tau_3) g_1(\xi_1,\tau_1) 
g_2(\xi_2,\tau_2) d\xi_1 d\xi_2 d\xi_3 d\tau_1 d\tau_2 d\tau_3 \, ,
\end{equation}
where * denotes the region $\{ \sum_{i=1}^3 \xi_i = \sum_{i=1}^3 \tau_i = 0 \}$ 
we have to show
$$|I(\widehat{n},\widehat{u}_1,\widehat{u}_2)| \lesssim 
\|u_1\|_{X^{0,\frac{5}{12}+}} 
\|u_2\|_{X^{0,\frac{5}{12}+}}\|n\|_{X_{\pm}^{-\frac{1}{2},\frac{5}{12}+}} \, .
 $$
We use dyadic decompositions
$$ u_k = \sum_{N_k,L_k \ge 1} S_{L_k} P_{N_k} u_k \, , \, n = \sum_{N,L \ge 1} 
W^{\pm}_L P_N n \, . $$
Defining
$$ g_k^{L_k,N_k} = {\mathcal F} S_{L_k} P_{N_k} u_k \, , \, f^{L,N} = {\mathcal 
F} 
W^{\pm}_L P_N n $$
we have
$$ I(\widehat{n},\widehat{u}_1,\widehat{u}_2) = \sum_{N,N_1,N_2 \ge 1 \,,\, 
L,L_1,L_2 \ge 1} I(f^{L,N},g_1^{L_1,N_1},g_2^{L_2,N_2}) \, . $$
{\bf Case 1:} $ N_1 \sim N_2 \gtrsim N \ge 2^{10} $ . \\
Fix $M=2^{-4} N_1$ and decompose
\begin{eqnarray}
\nonumber
\lefteqn{ I(f^{L,N},g_1^{L_1,N_1},g_2^{L_2,N_2}) } \\
\label{HHL}
& = & \sum_{0 \le j_1,j_2 \le M-1 \, ,\, |j_1-j_2|\le 16} 
I(f^{L,N},g_1^{L_1,N_1,M,j_1}, g_2^{L_2,N_2,M,j_2}) \\
\nonumber
& & + \sum_{64 \le A \le M} \,\, \sum_{0 \le j_1,j_2 \le A-1 \, ,\, 16 \le 
|j_1-j_2| \le 32} I(f^{L,N},g_1^{L_1,N_1,A,j_1}, g_2^{L_2,N_2,A,j_2}) \, .
\end{eqnarray}
The first sum is estimated using \cite[Prop. 4.7]{BHHT} by
$$ L_1^{\frac{5}{12}}  L_2^{\frac{5}{12}} L^{\frac{5}{12}} N^{-\frac{1}{2}} 
(\frac{N}{N_1})^{\frac{1}{4}} \|f^{L,N}\|_{L^2} \|g_1^{L_1,N_2}\|_{L^2} 
\|g_2^{L_2,N_2}\|_{L^2} \, . $$
The second sum is treated using \cite[Prop. 4.4 and Prop. 4.6]{BHHT} and $A \le 
M \ll N_1$. We distinguish two cases.\\
{\bf a.} $L,L_1,L_2 \le N_1^2$.\\
We define $ \alpha := 2^{-4} \min( (\frac{N_1}{N})^{\frac{1}{2}} N_1 
\max(L_1,L_2,L)^{-\frac{1}{2}} , N_1)$. The part where $A \le \alpha$ can be 
estimated for fixed $A$ using \cite[Prop. 4.4]{BHHT} by
$$ N_1^{-\frac{1}{2}} (\frac{A}{N_1})^{\frac{1}{2}} (L_1 L_2 L)^{\frac{1}{2}} 
\|f^{L,N}\|_{L^2} 
 \|g_1^{L_1,N_1,A,j_1}\|_{L^2} 
\|g_2^{L_2,N_2,A,j_2}\|_{L^2} \, . $$
 
Summing over $64 \le A \le \alpha$ and $j_1,j_2$ and using $\sum_{64 \le A \le 
\alpha} A^{\frac{1}{2}} \lesssim \alpha^{\frac{1}{2}}$ we get the bound
$$
N_1^{-\frac{1}{2}} (\frac{N_1}{N})^{\frac{1}{4}} (L_1 L_2 L)^{\frac{5}{12}} 
\|f^{L,N}\|_{L^2} 
 \|g_1^{L_1,N_1}\|_{L^2} 
\|g_2^{L_2,N_2}\|_{L^2} \, . $$ 
Next we consider the part $A \ge \alpha$. It is estimated using \cite[Prop. 
4.6]{BHHT} by
$$ N^{-\frac{1}{2}} (\frac{N_1}{A})^{\frac{1}{2}} (L_1 L_2 
L)^{\frac{1}{2}}\max(L_1,L_2,L)^{-\frac{1}{2}} \|f^{L,N}\|_{L^2} 
\|g_1^{L_1,N_1,A,j_1}\|_{L^2} 
\|g_2^{L_2,N_2,A,j_2}\|_{L^2} \, . $$
Summing over $\alpha \le A \le N_1$ and $j_1,j_2$ and using $\sum_{A \ge 
\alpha} 
A^{-\frac{1}{2}} \lesssim \alpha^{-\frac{1}{2}}$ we get the bound
\begin{eqnarray*}
\lefteqn{(L_1 L_2 L)^{\frac{1}{2}} \max(L_1,L_2,L)^{-\frac{1}{2}} 
N^{-\frac{1}{2}} N_1^{\frac{1}{2}} (\frac{N}{N_1})^{\frac{1}{4}} 
N_1^{-\frac{1}{2}} \max(L_1, L_2, L)^{\frac{1}{4}} \cdot} \\
& & \cdot \|f^{L,N}\|_{L^2} 
 \|g_1^{L_1,N_1}\|_{L^2} 
\|g_2^{L_2,N_2}\|_{L^2} \\
& &\le   N^{-\frac{1}{2}} (\frac{N}{N_1})^{\frac{1}{4}} (L_1 L_2 
L)^{\frac{5}{12}} \|f^{L,N}\|_{L^2} 
 \|g_1^{L_1,N_1}\|_{L^2} 
\|g_2^{L_2,N_2}\|_{L^2} \, .
\end{eqnarray*}
{\bf b.} $\max(L_1,L_2,L) \gtrsim N_1^2$.\\
\cite[Prop. 4.6]{BHHT} gives the following bound for fixed $A$:
\begin{eqnarray*}
\lefteqn{(L_1 L_2 L)^{\frac{1}{2}} \max(L_1,L_2,L)^{-\frac{1}{2}} 
N^{-\frac{1}{2}} (\frac{N_1}{A})^{\frac{1}{2}} \|f^{L,N}\|_{L^2} 
 \|g_1^{L_1,N_1,A,j_1}\|_{L^2} 
\|g_2^{L_2,N_2,A,j_2}\|_{L^2} } \\
& \lesssim & N^{-\frac{1}{2}} N_1^{-\frac{1}{2}-} (\frac{N_1}{A})^{\frac{1}{2}} 
(L_1 L_2 L)^{\frac{5}{12}+} \|f^{L,N}\|_{L^2} 
 \|g_1^{L_1,N_1,A,j_1}\|_{L^2} 
\|g_2^{L_2,N_2,A,j_2}\|_{L^2} \, .
\end{eqnarray*}
Summation over $64 \le A \le N_1$ and $j_1,j_2$ using $\sum A^{-\frac{1}{2}} 
\lesssim 1$ gives the bound
$$ N^{-\frac{1}{2}} N_1^{0-} (L_1 L_2 L)^{\frac{5}{12}+} \|f^{L,N}\|_{L^2} 
 \|g_1^{L_1,N_1}\|_{L^2} 
\|g_2^{L_2,N_2}\|_{L^2} \, . $$   
{\bf Case 2:} $ N_1 \ll N_2 $ or $ N_2 \ll N_1$.\\
Using \cite[Prop. 4.8]{BHHT} we get the bound
$$N^{-\frac{1}{2}} (L_1 L_2 L)^{\frac{5}{12}} 
\min(\frac{N_1}{N_2},\frac{N_2}{N_1})^{\frac{1}{6}} \|f^{L,N}\|_{L^2} 
 \|g_1^{L_1,N_1}\|_{L^2} 
\|g_2^{L_2,N_2}\|_{L^2} \, . $$ 
{\bf Case 3:} $N \lesssim 1$ ($\Rightarrow N_1 \sim N_2$ or $N_1,N_2 \lesssim 
1$).\\
\cite[Prop. 4.9]{BHHT} gives the bound
$$ (L_1 L_2 L)^{\frac{1}{3}} \|f^{L,N}\|_{L^2} 
 \|g_1^{L_1,N_1}\|_{L^2} 
\|g_2^{L_2,N_2}\|_{L^2} \, . $$
In any of these cases dyadic summation over $L_1,L_2,L$ and $N_1,N_2,N$ gives 
the desired bound. 
\end{proof}

\begin{prop}
\label{Prop. 2.2}
Assume $ s > - \frac{1}{2} \, , \, \sigma \ge - \frac{1}{2} \, , \, s < 
\sigma +1 \, . $ Then the following estimate holds:
$$ \| un \|_{X^{s,-\frac{1}{2}+}} 
\lesssim  
\|u\|_{X^{s,\frac{1}{2}-}} \|n\|_{X^{\sigma,\frac{1}{2}-}_{\pm}} \, . $$
\end{prop}
\begin{proof}  We have to show
$$|I(\widehat{n},\widehat{u}_1,\widehat{u}_2)| \lesssim  
\|u_1\|_{X^{s,\frac{1}{2}-}} 
\|u_2\|_{X^{-s,\frac{1}{2}-}}\|n\|_{X_{\pm}^{\sigma,\frac{1}{2}-}} \, .
 $$
Using dyadic decompositions as in the proof of Proposition \ref{Prop. 2.1} we 
consider different 
cases.\\
{\bf Case 1:} $N_1 \sim N_2 $.\\
This case can be treated by using Proposition \ref{Prop. 2.1} directly.\\
{\bf Case 2.} $ 1 \le N_1 \ll N_2$ $(\Rightarrow N \sim N_2)$.\\
We have
$$ L_{max}:=\max(L,L_1,L_2) \gtrsim |\tau_1 + |\xi_1|^2 + \tau_2 + |\xi_2|^2 + 
\tau_3 \pm |\xi_3|| = ||\xi_1|^2 + |\xi_2|^2 \pm |\xi_3|| \gtrsim N_2^2  . $$
Using the proof of \cite[Prop. 4.8]{BHHT} we consider three cases.\\
{\bf a.} $L=L_{max}$.\\
We get
\begin{eqnarray*}
\lefteqn{ |I(f^{L,N},g_1^{L_1,N_1},g_2^{L_2,N_2})| } \\
& \lesssim & L_1^{\frac{1}{2}} L_2^{\frac{1}{2}} 
(\frac{N_1}{N_2})^{\frac{1}{2}} 
\|f^{L,N}\|_{L^2} 
 \|g_1^{L_1,N_1}\|_{L^2} 
\|g_2^{L_2,N_2}\|_{L^2} \\
& \lesssim & (L_1 L_2 L)^{\frac{1}{2}-} N_2^{-1+} 
(\frac{N_1}{N_2})^{\frac{1}{2}} \|f^{L,N}\|_{L^2} 
 \|g_1^{L_1,N_1}\|_{L^2} 
\|g_2^{L_2,N_2}\|_{L^2} \, .
\end{eqnarray*}
{\bf b.} $L_1 = L_{max}$.\\
Similarly we get
\begin{eqnarray*}
\lefteqn{ |I(f^{L,N},g_1^{L_1,N_1},g_2^{L_2,N_2})| } \\
& \lesssim & L^{\frac{1}{2}} L_2^{\frac{1}{2}} (\frac{N_1}{N_2})^{\frac{1}{2}} 
\|f^{L,N}\|_{L^2} 
 \|g_1^{L_1,N_1}\|_{L^2} 
\|g_2^{L_2,N_2}\|_{L^2} \\
& \lesssim & (L_1 L_2 L)^{\frac{1}{2}-} N_2^{-1+} 
(\frac{N_1}{N_2})^{\frac{1}{2}} \|f^{L,N}\|_{L^2} \|g_1^{L_1,N_1}\|_{L^2} 
\|g_2^{L_2,N_2}\|_{L^2} \, .
\end{eqnarray*}
{\bf c.} $L_2 = L_{max}$.
\begin{eqnarray*}
|I(f^{L,N},g_1^{L_1,N_1},g_2^{L_2,N_2})| \lesssim  L^{\frac{1}{2}} 
L_1^{\frac{1}{2}} \|f^{L,N}\|_{L^2}  \|g_1^{L_1,N_1}\|_{L^2} 
\|g_2^{L_2,N_2}\|_{L^2} \\ \lesssim  (L_1 L_2 L)^{\frac{1}{2}-} N_2^{-1+}  
\|f^{L,N}\|_{L^2} 
 \|g_1^{L_1,N_1}\|_{L^2} 
\|g_2^{L_2,N_2}\|_{L^2} \, .
\end{eqnarray*}
If $ -\frac{1}{2} < s \le 0 $ and $ \sigma \ge -\frac{1}{2}$ we get
$$N_2^{-1+} \lesssim N_2^{s-} N_2^{\sigma -} \sim N_2^{s-} N^{\sigma -} 
\lesssim 
N_2^{0-} N_1^{s-} N^{\sigma -} \lesssim \frac{N_1^{s-}}{N_2^{s+}} N^{\sigma -} 
\, , $$
and in the case $s >0$ and $\sigma > s-1$ we get the same bound, because
$$ N_2^{-1+} \lesssim N_2^{-s+\sigma-} \lesssim \frac{N_2^{\sigma -}}{N_2^{s+}} 
\sim \frac{N^{\sigma -}}{N_2^{s+}} \lesssim \frac{N^{\sigma -}}{N_2^{s+}} 
N_1^{s-} \, . $$
In any case we thus get 
$$ |I(f^{L,N},g_1^{L_1,N_1},g_2^{L_2,N_2})| \lesssim \frac{N_1^{s-}}{N_2^{s+}} 
N^{\sigma -} (L_1 L_2 L)^{\frac{1}{2}-} \|f^{L,N}\|_{L^2}  
\|g_1^{L_1,N_1}\|_{L^2} 
\|g_2^{L_2,N_2}\|_{L^2} \, . $$
{\bf Case 3.} $1 \le N_2 \ll N_1$ ($\Rightarrow N \sim N_1$).\\
Similarly as in case 2 we get the bound
$$ |I(f^{L,N},g_1^{L_1,N_1},g_2^{L_2,N_2})| \lesssim  (L_1 L_2 
L)^{\frac{1}{2}-} 
N_1^{-1+}  \|f^{L,N}\|_{L^2} 
 \|g_1^{L_1,N_1}\|_{L^2} 
\|g_2^{L_2,N_2}\|_{L^2} 
\, . $$ 
If $-\frac{1}{2} < s \le 0$ and $\sigma \ge -\frac{1}{2}$ we get
$$ N_1^{-1+} \lesssim N_1^{s-} N^{\sigma -} \lesssim \frac{N_1^{s-}}{N_2^{s+}} 
N^{\sigma -} \, , $$
and if $s > 0$ and $\sigma > s-1$ we get
$$ N_1^{-1+} \lesssim \frac{N_1^{\sigma -}}{ N_1^{s+}} \sim \frac{N^{\sigma 
-}}{ 
N_1^{s+}}\lesssim \frac{N^{\sigma-}}{N_2^{s+}} \lesssim \frac{N^{\sigma 
-}}{N_2^{s+}} N_1^{s -} \, , $$
so that we get the same bound as in case 2.

Dyadic summation in all cases completes the proof of 
Prop. \ref{Prop. 2.2}.
\end{proof}
We also need the following bilinear estimate for our unconditional uniqueness 
result:
\begin{prop}
\label{Prop. 2.2'}
For any $\epsilon > 0$ the following estimate holds:
$$ \| un \|_{X^{-\epsilon,-\frac{1}{2}+}} 
\lesssim  
\|u\|_{X^{-\frac{1}{2}-,\frac{1}{2}+}} \|n\|_{X^{0-,\frac{1}{2}+}_{\pm}} \, . 
$$
\end{prop}
\begin{proof}
We use dyadic decompositions as in the proof of Proposition \ref{Prop. 2.1}.\\
{\bf Case 1:} $N_1 \sim N_2 \gtrsim N \ge 2^{10} $.\\
We use (\ref{HHL}). When estimating its first sum we consider different cases 
using the proof of \cite[Prop. 4.7]{BHHT}.\\
{\bf a.} $L=L_{max}$. \\
In this case we get the bound 
$$(L_1 L_2)^{\frac{1}{2}} N^{-\frac{1}{2}} \|f^{L,N}\|_{L^2}  
\|g_1^{L_1,N_1}\|_{L^2} 
\|g_2^{L_2,N_2}\|_{L^2}$$
and
\begin{itemize}
\item either $N \sim N_1$ in which case we have
$$(L_1 L_2)^{\frac{1}{2}} N^{-\frac{1}{2}} \le (L_1 L_2 L)^{\frac{1}{3}} 
N_1^{-\frac{1}{2}} \sim (L_1L_2L)^{\frac{1}{3}} 
N_1^{-\frac{1}{2}-}N_2^{0+}N^{0-} \, , $$
\item or $NN_1 \lesssim L_{max}$ in which case we get 
$$ (L_1 L_2)^{\frac{1}{2}} N^{-\frac{1}{2}} \lesssim 
(L_1L_2)^{\frac{1}{2}-}L^{\frac{1}{2}+} 
N^{-\frac{1}{2}}N_1^{-\frac{1}{2}-}N_2^{0+} \, . $$
\end{itemize}
{\bf b.} $L_1 = L_{max}$. \\
In this case we get the bound
\begin{eqnarray*}
\lefteqn{(L L_2)^{\frac{1}{2}} N_1^{-\frac{1}{2}} \|f^{L,N}\|_{L^2}  
\|g_1^{L_1,N_1}\|_{L^2} 
\|g_2^{L_2,N_2}\|_{L^2} } \\
& \lesssim & (LL_1L_2)^{\frac{1}{3}} N_1^{-\frac{1}{2}-} N_2^{0+} 
\|f^{L,N}\|_{L^2}  
\|g_1^{L_1,N_1}\|_{L^2} 
\|g_2^{L_2,N_2}\|_{L^2}  \, . 
\end{eqnarray*}
{\bf c.} $L_2 = L_{max}$. \\
This case is similar as case b.\\
The second sum in (\ref{HHL}) is estimated as follows.\\
{\bf a.} $L_{max} \lesssim N_1^2$. \\
By \cite[Prop. 4.4]{BHHT} for fixed $A$ we get the bound
\begin{eqnarray*}
\lefteqn{ N_1^{-\frac{1}{2}} (\frac{A}{N_1})^{\frac{1}{2}} (L_1 L_2 
L)^{\frac{1}{2}} 
\|f^{L,N}\|_{L^2} 
 \|g_1^{L_1,N_1,A,j_1}\|_{L^2} 
\|g_2^{L_2,N_2,A,j_2}\|_{L^2} } \\
& \lesssim & N_1^{-\frac{1}{2}+} (\frac{A}{N_1})^{\frac{1}{2}} (L_1 L_2 
L)^{\frac{1}{2}-} 
\|f^{L,N}\|_{L^2} 
 \|g_1^{L_1,N_1,A,j_1}\|_{L^2} 
\|g_2^{L_2,N_2,A,j_2}\|_{L^2}
\end{eqnarray*} 
Summing over $64 \le A \le 2^{-4}N_1$ and $j_1,j_2$ and using $\sum_{64 \le A 
\le 
N_1} A^{\frac{1}{2}} \lesssim N_1^{\frac{1}{2}}$ we get the bound
$$
N_1^{-\frac{1}{2}-} N_2^{0+} N^{0-} (L_1 L_2 L)^{\frac{1}{2}-} 
\|f^{L,N}\|_{L^2} 
 \|g_1^{L_1,N_1}\|_{L^2} 
\|g_2^{L_2,N_2}\|_{L^2} \, . $$ 
{\bf b.} $L_{max} \gtrsim N_1^2$. \\
By \cite[Prop. 4.6]{BHHT} for fixed $A$ we get the bound
\begin{eqnarray*}
\lefteqn{ N^{-\frac{1}{2}} (\frac{N_1}{A})^{\frac{1}{2}} \frac{(L_1 L_2 
L)^{\frac{1}{2}}}{L_{max}^{\frac{1}{2}}} 
\|f^{L,N}\|_{L^2} 
 \|g_1^{L_1,N_1,A,j_1}\|_{L^2} 
\|g_2^{L_2,N_2,A,j_2}\|_{L^2} } \\
& \lesssim & \frac{(L_1L_2 
L)^{\frac{1}{2}-}}{N_1^{\frac{1}{2}-}N^{\frac{1}{2}}} 
A^{-\frac{1}{2}} 
\|f^{L,N}\|_{L^2} 
 \|g_1^{L_1,N_1,A,j_1}\|_{L^2} 
\|g_2^{L_2,N_2,A,j_2}\|_{L^2}
\end{eqnarray*} 
Summing over $64 \le A \le 2^{-4}N_1$ and $j_1,j_2$ and using $\sum 
A^{-\frac{1}{2}} \lesssim 1$ we get the bound
$$
\frac{(L_1 L_2 L)^{\frac{1}{2}-}N_2^{0+}}{N_1^{\frac{1}{2}+}  N^{\frac{1}{2}}}  
\|f^{L,N}\|_{L^2} 
 \|g_1^{L_1,N_1}\|_{L^2} 
\|g_2^{L_2,N_2}\|_{L^2} \, . $$ 
{\bf Case 2.} $ 1 \le N_1 \ll N_2$ \\
Similarly as in the proof of Proposition \ref{Prop. 2.2} we have
\begin{eqnarray*}
\lefteqn{ |I(f^{L,N},g_1^{L_1,N_1},g_2^{L_2,N_2})| } \\
& \lesssim &   (L_1 L_2 L)^{\frac{1}{2}-} N_2^{-1+}
 \|f^{L,N}\|_{L^2} 
 \|g_1^{L_1,N_1}\|_{L^2} 
\|g_2^{L_2,N_2}\|_{L^2} 
\end{eqnarray*}
{\bf Case 3.} $ 1 \le N_1 \ll N_2$ \\
We have similarly as in case 2:
\begin{eqnarray*}
\lefteqn{ |I(f^{L,N},g_1^{L_1,N_1},g_2^{L_2,N_2})| } \\
& \lesssim &   (L_1 L_2 L)^{\frac{1}{2}-} N_1^{-1+}
 \|f^{L,N}\|_{L^2} 
 \|g_1^{L_1,N_1}\|_{L^2} 
\|g_2^{L_2,N_2}\|_{L^2}
\end{eqnarray*}
{\bf Case 4.} $ 1 \le N \lesssim 1 $ ($\Rightarrow N_1 \sim N_2$ or $ 1\le 
N_1,N_2,N \lesssim 1$)\\
By the bilinear Strichartz type estimate \cite[Prop. 4.3]{BHHT}
 we get
\begin{eqnarray*}
\lefteqn{ |I(f^{L,N},g_1^{L_1,N_1},g_2^{L_2,N_2})| } \\
& \lesssim & (\frac{\min(N,N_2)}{N_2})^{\frac{1}{2}}  (L_2 L)^{\frac{1}{2}} 
 \|f^{L,N}\|_{L^2} 
 \|g_1^{L_1,N_1}\|_{L^2} 
\|g_2^{L_2,N_2}\|_{L^2} \\
& \lesssim & N_2^{-\frac{1}{2}}  (L_2 L)^{\frac{1}{2}} 
 \|f^{L,N}\|_{L^2} 
 \|g_1^{L_1,N_1}\|_{L^2} 
\|g_2^{L_2,N_2}\|_{L^2}\\
& \lesssim & N_2^{-\frac{1}{2}-} N_1^{0+}  (L_2 L)^{\frac{1}{2}} 
 \|f^{L,N}\|_{L^2} 
 \|g_1^{L_1,N_1}\|_{L^2} 
\|g_2^{L_2,N_2}\|_{L^2} \, .
\end{eqnarray*}
Dyadic summation in all cases completes the proof of 
Prop. \ref{Prop. 2.2'}.
\end{proof}  
\begin{prop}
\label{Prop. 2.3}
Assume $ s > - \frac{1}{4} \, , \, \sigma  - 2s <\frac{3}{2} \, , \, \sigma < 
s+2 \, . $ Then the following estimate holds:
$$ \| u_1 \overline{u}_2 \|_{X_{\pm}^{\sigma -1,-\frac{1}{2}+}} 
\lesssim  
\|u_1\|_{X^{s,\frac{1}{2}-}} \|u_2\|_{X^{s,\frac{1}{2}-}} \, . $$
\end{prop}
\begin{proof}
With $I$ defined by (\ref{I}) we have to show
$$|I(\widehat{n},\widehat{u}_1,\widehat{u}_2)| \lesssim 
\|u_1\|_{X^{s,\frac{1}{2}-}} 
\|u_2\|_{X^{s,\frac{1}{2}-}}\|n\|_{X_{\pm}^{1-\sigma,\frac{1}{2}-}} \, .
 $$
Dyadically decomposing as in Proposition \ref{Prop. 2.2} we consider different 
cases.\\
{\bf Case 1.} $ N_1 \sim N_2 \gtrsim N \ge 2^{10}$. Fix $M=2^{-4} N_1$ and use 
(\ref{HHL}).\\
{\bf Case 1.1.} $ L_{max} \lesssim N_1^2$.\\
Using \cite[Prop. 4.4]{BHHT} for the second sum in (\ref{HHL}) we get the bound
\begin{eqnarray*}
\lefteqn{N_1^{-\frac{1}{2}} (L_1 L_2 L)^{\frac{1}{2}} \|f^{L,N}\|_{L^2} 
\sum_{64 
\le A \le M} (\frac{A}{N_1})^{\frac{1}{2}} } \\
& & \sum_{0 \le j_1,j_2 \le A-1\, , \, 16 \le |j_1 - j_2| \le 32} 
\|g_1^{L_1,N_1,A,j_1}\|_{L^2} \|g_2^{L_2,N_2,A,j_2}\|_{L^2} \\
& \lesssim & N_1^{-\frac{1}{2}+} (L_1 L_2 L)^{\frac{1}{2}-} \|f^{L,N}\|_{L^2} 
\|g_1^{L_1,N_1}\|_{L^2} \|g_2^{L_2,N_2}\|_{L^2} \, . 
\end{eqnarray*}
For the first sum in (\ref{HHL}) we use the proof of \cite[Prop. 4.7]{BHHT}, 
which gives \\
{\bf a.} in the case $L = L_{max}$
\begin{itemize}
\item either $N \sim N_1$ and thus the bound
\begin{eqnarray*}
\lefteqn{ I(f^{L,N},g_1^{L_1,N_1,M,j_1},g_2^{L_2,N_2,M,j_2}) } \\
& \lesssim & \frac{N_1^{\frac{1}{2}}}{M^{\frac{1}{2}}} \frac{(L_1 
L_2)^{\frac{1}{2}}}{N^{\frac{1}{2}}} \|f^{L,N}\|_{L^2} 
\|g_1^{L_1,N_1,M,j_1}\|_{L^2} \|g_2^{L_2,N_2,M,j_2}\|_{L^2} \\
& \lesssim & \frac{(L_1 L_2 L)^{\frac{1}{2}-}}{N_1^{\frac{1}{2}}} 
\|f^{L,N}\|_{L^2} \|g_1^{L_1,N_1,M,j_1}\|_{L^2} \|g_2^{L_2,N_2,M,j_2}\|_{L^2} 
\, 
,
\end{eqnarray*}
\item or $NN_1 \lesssim L_{max}$ and thus
\begin{eqnarray*}
\lefteqn{ I(f^{L,N},g_1^{L_1,N_1,M,j_1},g_2^{L_2,N_2,M,j_2}) } \\
& \lesssim & \frac{N_1^{\frac{1}{2}}}{M^{\frac{1}{2}}} \frac{(L_1 
L_2)^{\frac{1}{2}}}{N^{\frac{1}{2}}} \|f^{L,N}\|_{L^2} 
\|g_1^{L_1,N_1,M,j_1}\|_{L^2} \|g_2^{L_2,N_2,M,j_2}\|_{L^2} \\
& \lesssim & \frac{(L_1 L_2 L)^{\frac{1}{2}-}}{N^{\frac{1}{2}} N^{\frac{1}{2}-} 
N_1^{\frac{1}{2}-}} \|f^{L,N}\|_{L^2} \|g_1^{L_1,N_1,M,j_1}\|_{L^2} 
\|g_2^{L_2,N_2,M,j_2}\|_{L^2} \, .
\end{eqnarray*}
\end{itemize}
{\bf b.} In the case $L_1 = L_{max}$ we get the bound
\begin{eqnarray*}
\lefteqn{ I(f^{L,N},g_1^{L_1,N_1,M,j_1},g_2^{L_2,N_2,M,j_2}) } \\
& \lesssim & \frac{(L L_2)^{\frac{1}{2}}}{M^{\frac{1}{2}}} \|f^{L,N}\|_{L^2} 
\|g_1^{L_1,N_1,M,j_1}\|_{L^2} \|g_2^{L_2,N_2,M,j_2}\|_{L^2} \\
& \lesssim & \frac{(L_1 L_2 L)^{\frac{1}{2}-}}{N_1^{\frac{1}{2}}} 
\|f^{L,N}\|_{L^2} \|g_1^{L_1,N_1,M,j_1}\|_{L^2} \|g_2^{L_2,N_2,M,j_2}\|_{L^2} 
\, 
.
\end{eqnarray*}
{\bf c.} The case $L_2 = L_{max}$ is similar.\\
Thus the first sum in (\ref{HHL}) can be bounded like the second sum.\\
{\bf Case 1.2.} $L_{max} \gtrsim N_1^2$. \\
The first sum in (\ref{HHL}) is treated exactly as before, whereas the second 
sum 
is estimated using \cite[Prop. 4.6]{BHHT} by 
\begin{eqnarray*}
\lefteqn{ \sum_{64 \le A \le M} \frac{(L_1 L_2 L)^{\frac{1}{2}} 
N^{-\frac{1}{2}}}{L_{max}^{\frac{1}{2}}} (\frac{N_1}{A})^{\frac{1}{2}} } \\ & 
&\sum_{0 \le j_1,j_2 \le A-1\, , \, 16 \le |j_1-j_2| \le 32} 
I(f^{L,N},g_1^{L_1,N_1,A,j_1},g_2^{L_2,N_2,A,j_2}) \\
& \lesssim & \frac{(L_1 L_2 L)^{\frac{1}{2}-}}{N^{\frac{1}{2}} 
N_1^{\frac{1}{2}-}} \| f^{L,N}\|_{L^2}  \| g_1^{L_1,N_1}\|_{L^2} \| 
g_2^{L_2,N_2}\|_{L^2} \, ,
\end{eqnarray*}
where we used the estimate
$$ \frac{(L_1 L_2 L)^{\frac{1}{2}} N^{-\frac{1}{2}}}{L_{max}^{\frac{1}{2}}} 
(\frac{N_1}{A})^{\frac{1}{2}} \lesssim \frac{(L_1 L_2 L)^{\frac{1}{2}-} 
}{N^{\frac{1}{2}} N_1^{1-}} \frac{N_1^{\frac{1}{2}}}{A^{\frac{1}{2}}} $$
and $ \sum_{A} A^{-\frac{1}{2}} \lesssim 1 $ .

Summarizing, we get
\begin{eqnarray*}
\lefteqn{ |I(f^{L,N},g_1^{L_1,N_1},g_2^{L_2,N_2})| } \\
& \lesssim & (L_1 L_2 L)^{\frac{1}{2}-} N_1^{-\frac{1}{2}+} \| f^{L,N}\|_{L^2}  
\| g_1^{L_1,N_1}\|_{L^2} \| g_2^{L_2,N_2}\|_{L^2} \\
& \lesssim & (L_1 L_2 L)^{\frac{1}{2}-} N_1^{s-} N_2^{s-} N^{1-\sigma -} \| 
f^{L,N}\|_{L^2}  \| g_1^{L_1,N_1}\|_{L^2} \| g_2^{L_2,N_2}\|_{L^2} \, ,
\end{eqnarray*}
where we used $ s > - \frac{1}{4}$ and $\sigma < 2s + \frac{3}{2}$ to get
$$ N_1^{-\frac{1}{2}+} \lesssim N_1^{s-} N_2^{s-} N_1^{-\frac{1}{2}-2s+} 
\lesssim N_1^{s-} N_2^{s-} N^{-\frac{1}{2}-2s+} \lesssim N_1^{s-} N_2^{s-} 
N^{1-\sigma -} \, . $$
Dyadic summation over $N_1,N_2,N$ and $L_1,L_2,L$ gives the claimed estimate.\\
{\bf Case 2.} $N_1 \ll N_2 \sim N $ (or similarly $N_2 \ll N_1 \sim N$). \\
As in the proof of Prop. \ref{Prop. 2.2} we get the bound
\begin{eqnarray*}
\lefteqn{ |I(f^{L,N},g_1^{L_1,N_1},g_2^{L_2,N_2})| } \\
& \lesssim & (L_1 L_2 L)^{\frac{1}{2}-} N_2^{-1+} \| f^{L,N}\|_{L^2}  \| 
g_1^{L_1,N_1}\|_{L^2} \| g_2^{L_2,N_2}\|_{L^2} \\
& \lesssim & (L_1 L_2 L)^{\frac{1}{2}-} N_1^{s-} N_2^{s-} N^{1-\sigma -} \| 
f^{L,N}\|_{L^2}  \| g_1^{L_1,N_1}\|_{L^2} \| g_2^{L_2,N_2}\|_{L^2} \, ,
\end{eqnarray*}
where we used $\sigma < s+2$ to get in the case $s>0$
$$N_2^{-1+} \lesssim N_2^{s+1-\sigma-} \lesssim N_1^{s-} N_2^{s-} N^{1-\sigma 
-} 
$$
and $\sigma < 2s + \frac{3}{2}$ to get in the case $ s \le 0$
$$ N_2^{-1+} \lesssim N_2^{-2s-1+} N_2^{s-} N_2^{s-} \lesssim 
N_2^{\frac{1}{2}-\sigma-} N_1^{s-} N_2^{s-} \lesssim N^{\frac{1}{2}-\sigma -} 
N_1^{s-} N_2^{s-} \, ,
$$
which is more than enough to get the claimed result after dyadic summation.\\
{\bf Case 3.} $N \lesssim 1$ ($\Rightarrow N_1 \sim N_2$ or $N,N_1,N_2 \lesssim 
1$).\\
Assuming without loss of generality $L_1 \le L_2$ and using the bilinear 
Strichartz type estimate \cite[Prop. 4.3]{BHHT} we get
\begin{eqnarray*}
\lefteqn{ |I(f^{L,N},g_1^{L_1,N_1},g_2^{L_2,N_2})| } \\
& \lesssim & \| f^{L,N} g_1^{L_1,N_1}\|_{L^2} \|g_2^{L_2,N_2}\|_{L^2} \\
& \lesssim & \min(N,N_1)^{\frac{1}{2}} N_1^{-\frac{1}{2}} L^{\frac{1}{2}} 
L_1^{\frac{1}{2}} \| f^{L,N}\|_{L^2} \| g_1^{L_1,N_1}\|_{L^2} 
\|g_2^{L_2,N_2}\|_{L^2} \\
& \lesssim & N_1^{-\frac{1}{2}} L^{\frac{1}{2}} \| f^{L,N}\|_{L^2} 
L_1^{\frac{1}{4}} \| g_1^{L_1,N_1}\|_{L^2} L_2^{\frac{1}{4}} \| 
g_2^{L_2,N_2}\|_{L^2} \, . 
\end{eqnarray*}
Furthermore we get by \cite[formula (4.22)]{BHHT}
$$ |I(f^{L,N},g_1^{L_1,N_1},g_2^{L_2,N_2})| \le L^{\frac{1}{3}} \| 
f^{L,N}\|_{L^2}  L_1^{\frac{1}{3}} \| g_1^{L_1,N_1}\|_{L^2} 
L_2^{\frac{1}{3}} \| g_2^{L_2,N_2}\|_{L^2} \, , $$
so that by interpolation we arrive at
\begin{eqnarray*}
\lefteqn{ |I(f^{L,N},g_1^{L_1,N_1},g_2^{L_2,N_2})| } \\
& \lesssim & N_1^{-\frac{1}{2}+} L^{\frac{1}{2}-} \| f^{L,N}\|_{L^2} 
L_1^{\frac{1}{4}+} \| g_1^{L_1,N_1}\|_{L^2} L_2^{\frac{1}{4}+} \| 
g_2^{L_2,N_2}\|_{L^2} \\
& \lesssim & N_1^{s-} N_2^{s-} N^{1-\sigma -} L^{\frac{1}{2}-} \| 
f^{L,N}\|_{L^2} L_1^{\frac{1}{4}+} \| g_1^{L_1,N_1}\|_{L^2} L_2^{\frac{1}{4}+} 
\| g_2^{L_2,N_2}\|_{L^2} \, , 
\end{eqnarray*}
using $ s > - \frac{1}{4}$ . Dyadic summation again gives the claimed result.
\end{proof}

\begin{proof} [Proof of Theorem \ref{Theorem 2.1}] It is by now standard to use 
Proposition \ref{Prop. 2.2} and Proposition 
\ref{Prop. 2.3} to show the local well-posedness result (Theorem \ref{Theorem 
2.1}) for the 
system (\ref{0.1'}),(\ref{0.2'}),(\ref{0.3'}) as an application of the 
contraction mapping principle. For details 
of the method we refer to \cite{GTV}. This solution then immediately leads to a 
solution of the Klein-Gordon-Schr\"odinger system 
(\ref{0.1}),(\ref{0.2}),(\ref{0.3}) with the required properties as explained 
before Theorem \ref{Theorem 2.1}.

Moreover, if $(u,n,\partial_t n)$ is a solution of (the system of integral 
equations belonging to) (\ref{0.1}),(\ref{0.1}),(\ref{0.3}) with $u \in 
X^{s,\frac{1}{2}+}[0,T]$ and data $u_0 \in H^s$, $n_0 \in H^{\sigma}$, $n_1 \in 
H^{\sigma}$, then $n_{\pm}$ defined by (\ref{9}) belongs to 
$X_{\pm}^{\sigma,\frac{1}{2}+}[0,T]$ by Proposition \ref{Prop. 2.3} and thus 
$n= 
\frac{1}{2}(n_+ + n_-)$ belongs to $X_+^{\sigma,\frac{1}{2}+}[0,T] + 
X_-^{\sigma,\frac{1}{2}+}[0,T]$ and 
$\partial_t n = \frac{1}{2i} A^{\frac{1}{2}}(n_+ - n_-)$ belongs to $ 
X_+^{\sigma -1,\frac{1}{2}+,}[0,T] + X_-^{\sigma -1,\frac{1}{2}+}[0,T]$ , and
one easily checks that $(u,n_+,n_-)$ is a solution of the system (of integral 
equations belonging to) 
(\ref{0.1'}),(\ref{0.2'}),(\ref{0.3'}). But because this solution is uniquely 
determined the solution of the Klein - Gordon - Schr\"odinger system is also 
unique.

For the part concerning unconditional uniqueness we use an idea 
of Y. Zhou \cite{Z},\cite{Z1}, which we already applied in \cite[Prop. 3.1]{P}.
Let
$$ (u,n,\partial_t n) \in C^0([0,T],L^2({\mathbb R}^2)) \times  
C^0([0,T],L^2({\mathbb R}^2)) 
\times C^0([0,T],H^{-1}({\mathbb R}^2)) $$ 
be any solution of the Klein-Gordon-Schr\"odinger 
system (\ref{0.1}),(\ref{0.2}),(\ref{0.3}). This leads to a corresponding 
solution of 
the system (\ref{0.1'}),(\ref{0.2'}),(\ref{0.3'}) with $$ (u,n_+,n_-) \in 
C^0([0,T],L^2({\mathbb R}^2)) \times  C^0([0,T],L^2({\mathbb 
R}^2)) 
\times C^0([0,T],L^2({\mathbb R}^2)) \, .$$ 
By Sobolev's embedding theorem we get
\begin{eqnarray*}
\| n_{\pm} u\|_{L^2((0,T),H^{-1-})} & \lesssim & \|n_{\pm} u 
\|_{L^2((0,T),L^1)} \\
& \lesssim & T^{\frac{1}{2}} \|n_{\pm}\|_{L^{\infty}((0,T);L^2)} 
\|u\|_{L^{\infty}((0,T),L^2)} < \infty \, . 
\end{eqnarray*}
so that from (\ref{0.1'}) we have $ u \in X^{-1-,1}[0,T]$ , because
$$ \|(i\partial_t + \Delta)u\|^2_{L^2((0,T),H^{-1-})} + 
\|u\|^2_{L^2((0,T),H^{-1-})} \sim \|u\|^2_{X^{-1-,1}[0,T]} 
< \infty \, . $$
Interpolation with $u \in X^{0,0}[0,T]$ gives 
$ u \in X^{-\frac{1}{2}-, \frac{1}{2}+}[0,T].$ 
Similarly we get
$$ \| |u|^2 \|_{L^2((0,T),H^{-1-})} \lesssim T^{\frac{1}{2}} 
\|u\|^2_{L^{\infty}((0,T),L^2)} < \infty $$
and from 
(\ref{0.2'}) we conclude $n_{\pm} \in X_{\pm}^{0-,1}[0,T]$ .
Proposition \ref{Prop. 2.2'} shows that $un_{\pm}$ $ \in 
X^{-\epsilon,-\frac{1}{2}+}[0,T]$, thus $u\in X^{-\epsilon,\frac{1}{2}+}[0,T]$ 
and $n_{\pm} \in X_{\pm}^{0-,1}[0,T]$ for any $\epsilon > 0.$ 
But in these spaces uniqueness holds by the first part of this proof, so that 
unconditional uniqueness is also proven.
\end{proof} 

\section{Global well-posedness results for the case $D=2$}
We first show a modified local well-posedness result in arbitrary space 
dimension $D$.
\begin{prop}
\label{Prop. 3.1}
Let $u_0 \in L^2({\mathbb R}^D) \, , \, n_0 \in H^{\sigma}({\mathbb R}^D) \, , 
\, n_1 \in H^{\sigma -1}({\mathbb R}^D)$ and $T \le 1$. Assume
\begin{equation}
\label{12}
\|un\|_{X^{0,-\frac{1}{2}+}} \lesssim T^l \|u\|_{X^{0,\frac{1}{2}}} 
\|n\|_{X_{\pm}^{\sigma,\frac{1}{2}}}
\end{equation}
and
\begin{equation}
\label{13}
\| |u|^2 \|_{X_{\pm}^{\sigma -1,-\frac{1}{2}+}} \lesssim T^k 
\|u\|_{X^{0,\frac{1}{2}}}^2 \, ,
\end{equation}
where $k,l > 0$. \\
Then there exists $ 1 \ge T > 0 $ such that the system of integral equations
(\ref{8}),(\ref{9}) has a unique solution $u \in X^{0,\frac{1}{2}+}[0,T] $ , $ 
n_{\pm} \in X_{\pm}^{\sigma,\frac{1}{2}+}[0,T]$ . \\
 $n_{\pm}$ fulfills for $0 \le 
t \le T$ :
\begin{equation}
\label{13'}
\|n_{\pm}(t)\|_{H^{\sigma}} \le \|n_{\pm 0}\|_{H^{\sigma}} + cT^k 
\|u_0\|_{L^2}^2 
\, .
\end{equation}
$T$ can be chosen such that
\begin{eqnarray}
\label{14}
T^l (\|n_{+ 0}\|_{H^{\sigma}} + \|n_{- 0}\|_{H^{\sigma}}) & \lesssim & 1 \\
\label{15}
T^l \|u_0\|_{L^2} & \lesssim & 1 \\
\label{15'}
T^k \|u_0\|_{L^2} & \lesssim & 1 \\
\label{16}
T^k \|u_0\|_{L^2}^2 & \lesssim & \|n_{+ 0}\|_{H^{\sigma}} + \|n_{- 
0}\|_{H^{\sigma}} \, .
\end{eqnarray}
\end{prop}
\noindent Remark: No implicit constant appears on the right hand side of 
(\ref{13'}).
\begin{proof}
We construct a fixed point of $S=(S_0,S_+,S_-)$ in 
\begin{eqnarray*}
M:& = & \{ u \in X^{0,\frac{1}{2}+}[0,T] \, , \, n_{\pm} \in 
X_{\pm}^{\sigma,\frac{1}{2}+}[0,T] : \\
& & \, \|u\|_{X^{0,\frac{1}{2}+}} \lesssim \|u_0\|_{L^2} \, , \, 
\|n_+\|_{X_+^{\sigma,\frac{1}{2}+}} + \|n_-\|_{X_-^{\sigma,\frac{1}{2}+}} 
\lesssim \|n_{+0}\|_{H^{\sigma}} + \|n_{-0}\|_{H^{\sigma}} \}  , 
\end{eqnarray*}
where $S_0 u$ and $S_{\pm} n_{\pm}$ denote the right hand sides of our integral 
equations (\ref{8}) and (\ref{9}). Then we get for $u,n_{\pm} \in M$:
\begin{eqnarray*}
\|S_0 u\|_{X^{0,\frac{1}{2}+}} & \lesssim & \|u_0\|_{L^2} + T^l 
\|u\|_{X^{0,\frac{1}{2}+}} (\|n_+\|_{X^{\sigma,\frac{1}{2}+}_+} + 
\|n_-\|_{X^{\sigma,\frac{1}{2}+}_+}) \\
& \lesssim & \|u_0\|_{L^2} + T^l \|u_0\|_{L^2} (\|n_{+0}\|_{H^{\sigma}} + 
\|n_{-0}\|_{H^{\sigma}}) \quad \lesssim \quad \|u_0\|_{L^2}
\end{eqnarray*}
by (\ref{14}), and
\begin{eqnarray*}
\|S_+ n_+\|_{X^{\sigma,\frac{1}{2}+}_+} + \|S_- 
n_-\|_{X^{\sigma,\frac{1}{2}+}_-} & \lesssim & \|n_{+0}\|_{H^{\sigma}} +  
\|n_{-0}\|_{H^{\sigma}} + T^k \|u\|_{X^{0,\frac{1}{2}+}}^2 \\
& \lesssim & \|n_{+0}\|_{H^{\sigma}} +  \|n_{-0}\|_{H^{\sigma}} + T^k 
\|u_0\|_{L^2}^2 \\
& \lesssim & \|n_{+0}\|_{H^{\sigma}} +  \|n_{-0}\|_{H^{\sigma}}
\end{eqnarray*}
by (\ref{16}), such that $ S: M \to M $ .

In order to show the contraction property we estimate as follows. For 
$(u,n_{\pm}),$ $ (\tilde{u},\tilde{n}_{\pm}) \in M$ we get
\begin{eqnarray*}
\lefteqn{\|S_0 u - S_0 \tilde{u}\|_{X^{0,\frac{1}{2}+}} } \\
& \lesssim & T^l ( \|u-\tilde{u}\|_{X^{0,\frac{1}{2}+}} 
(\|n_+\|_{X^{\sigma,\frac{1}{2}+}_+} + 
\|\tilde{n}_+\|_{X^{\sigma,\frac{1}{2}+}_+} + 
\|n_-\|_{X^{\sigma,\frac{1}{2}+}_-} + 
\|\tilde{n}_-\|_{X^{\sigma,\frac{1}{2}+}_-}) \\
& & + (\|u\|_{X^{0,\frac{1}{2}+}} + \|\tilde{u}\|_{X^{0,\frac{1}{2}+}})(\|n_+ - 
\tilde{n}_+\|_{X_+^{\sigma,\frac{1}{2}+}} + \|n_- - 
\tilde{n}_-\|_{X_-^{\sigma,\frac{1}{2}+}})) \\
& \lesssim & T^l ( \|u-\tilde{u}\|_{X^{0,\frac{1}{2}+}} 
(\|n_{+0}\|_{H^{\sigma}} 
+ \|n_{-0}\|_{H^{\sigma}}  \\
& & + \|u_0\|_{L^2} (\|n_+ - \tilde{n}_+\|_{X_+^{\sigma,\frac{1}{2}+}} + \|n_- 
- 
\tilde{n}_-\|_{X_-^{\sigma,\frac{1}{2}+}})) \\
& \le & \frac{1}{2} (\|u-\tilde{u}\|_{X^{0,\frac{1}{2}+}} + \|n_+ - 
\tilde{n}_+\|_{X_+^{\sigma,\frac{1}{2}+}} + \|n_- - 
\tilde{n}_-\|_{X_-^{\sigma,\frac{1}{2}+}})
\end{eqnarray*}
by (\ref{14}) and (\ref{15}). Similarly
\begin{eqnarray*}
\lefteqn{\|S_+ n_+ - S_+ \tilde{n}_+\|_{X_+^{\sigma,\frac{1}{2}+}} + \|S_- n_- 
- 
S_- \tilde{n}_-\|_{X_-^{\sigma,\frac{1}{2}+}} ,} \\
& \lesssim & T^k (\|u\|_{X^{0,\frac{1}{2}+}} + 
\|\tilde{u}\|_{X^{0,\frac{1}{2}+}}) \|u - \tilde{u}\|_{X^{0,\frac{1}{2}+}} \\
& \lesssim & T^k \|u_0\|_{L^2}  \|u - \tilde{u}\|_{X^{0,\frac{1}{2}+}} \\
& \le & \frac{1}{2}  \|u - \tilde{u}\|_{X^{0,\frac{1}{2}+}}
\end{eqnarray*}
by (\ref{15'}). Thus the contraction mapping principle gives a unique solution 
in $[0,T]$. This solution fulfills $\|u(t)\|_{L^2} = \|u_0\|_{L^2}$.

Moreover we get from the integral equations (\ref{9}) for $0 \le t \le T$:
$$ \|n_{\pm}(t)\|_{H^{\sigma}} \le \|n_{\pm 0}\|_{H^{\sigma}} + cT^k 
\|u\|_{X^{0,\frac{1}{2}+}}^2 \le \|n_{\pm 0}\|_{H^{\sigma}} + cT^k 
\|u_0\|_{L^2}^2 \, , $$
using that $e^{\mp itA^{\frac{1}{2}}}$ is unitary.
\end{proof}
This version of the local well-posedness result will now be used to show the 
global well-posedness result Theorem \ref{Theorem 3.2}. We first 
show
\begin{prop}
\label{Prop. 3.2}
In space dimension $D=2$ assume $-\frac{1}{2} \le \sigma < \frac{3}{2}$ and $T 
\le 1$. Then the estimates (\ref{12}) and (\ref{13}) hold with $ 1 > k,l \ge 
\frac{1}{4}-$ and $k+l \ge \frac{5}{4}-$.
\end{prop}
\begin{proof}
We estimate $I$ (defined by (\ref{I})) by H\"older's inequality and Sobolev's 
embedding:
$$ |I(\widehat{n},\widehat{u}_1,\widehat{u}_2)| \lesssim \|n\|_{L_t^3 
L_x^{\infty}} \|u_1\|_{L_t^3 L_x^2} \|u_2\|_{L_t^3 L_x^2} \lesssim 
\|n\|_{X^{1+,\frac{1}{6}}_{\pm}} \|u_1\|_{X^{0,\frac{1}{6}}} 
\|u_2\|_{X^{0,\frac{1}{6}}} \, . $$
Thus
$$ \|un\|_{X^{0,-\frac{1}{6}}} \lesssim \|u\|_{X^{0,\frac{1}{6}}} 
\|n\|_{X_{\pm}^{1+,\frac{1}{6}}} \, . $$
This implies by Lemma \ref{Lemma}:
\begin{eqnarray*}
\lefteqn{\|un\|_{X^{0,-\frac{1}{2}+}}} \\
& \lesssim & T^{\frac{1}{3}-} \|un\|_{X^{0,-\frac{1}{6}}} \lesssim 
T^{\frac{1}{3}-} \|u\|_{X^{0,\frac{1}{6}}} \|n\|_{X_{\pm}^{1+,\frac{1}{6}}} 
\lesssim T^{1-} \|u\|_{X^{0,\frac{1}{2}-}} \|n\|_{X_{\pm}^{1+,\frac{1}{2}-}} 
\,. 
\end{eqnarray*} 
Moreover by Proposition \ref{Prop. 2.1}:
\begin{eqnarray*}
\|un\|_{X^{0,-\frac{1}{2}+}} 
& \lesssim & T^{\frac{1}{12}-} \|un\|_{X^{0,-\frac{5}{12}-}} \\ 
&\lesssim & T^{\frac{1}{12}-} \|u\|_{X^{0,\frac{5}{12}+}} 
\|n\|_{X_{\pm}^{-\frac{1}{2},\frac{5}{12}+}} \lesssim T^{\frac{1}{4}-} 
\|u\|_{X^{0,\frac{1}{2}-}} \|n\|_{X_{\pm}^{1+,\frac{1}{2}-}} \,.
\end{eqnarray*}
Interpolation gives for $ 0 \le \theta \le 1$ :
$$
\|un\|_{X^{0,-\frac{1}{2}+}} 
\lesssim  T^{1-\frac{3}{4}\theta -} \|u\|_{X^{0,\frac{1}{2}-}} 
\|n\|_{X_{\pm}^{1-\frac{3}{2}\theta,\frac{1}{2}-}} \,.
$$
By duality we also get
$$ \| |u|^2\|_{X_{\pm}^{\frac{3}{2}\theta -1,-\frac{1}{2}+}} \lesssim 
T^{1-\frac{3}{4}\theta-} \|u\|^2_{X^{0,\frac{1}{2}-}} \, . $$
\begin{itemize}
\item If $-\frac{1}{2}\le \sigma \le 1$ we choose $\theta = 
\frac{2}{3}(1-\sigma)$ and get
$$
\|un\|_{X^{0,-\frac{1}{2}+}} 
\lesssim  T^{\frac{1}{2}+\frac{\sigma}{2} -} \|u\|_{X^{0,\frac{1}{2}-}} 
\|n\|_{X_{\pm}^{\sigma,\frac{1}{2}-}} \,.
$$
\item If $1 \le \sigma < \frac{3}{2}$ we choose $\theta = 0$ and get
$$
\|un\|_{X^{0,-\frac{1}{2}+}} 
\lesssim  T^{1 -} \|u\|_{X^{0,\frac{1}{2}-}} \|n\|_{X_{\pm}^{1,\frac{1}{2}-}}
\lesssim  T^{1 -} \|u\|_{X^{0,\frac{1}{2}-}} 
\|n\|_{X_{\pm}^{\sigma,\frac{1}{2}-}} \,.
$$ 
\item If $-\frac{1}{2} \le \sigma \le 0$ we choose $\theta = 0$ and get
$$\| |u|^2\|_{X_{\pm}^{\sigma -1,-\frac{1}{2}+}} \lesssim \| |u|^2 
\|_{X_{\pm}^{ 
-1,-\frac{1}{2}+}} \lesssim T^{1-} \|u\|^2_{X^{0,\frac{1}{2}-}} \, . $$
\item If $0 \le \sigma < \frac{3}{2}$ we choose $\theta = \frac{2}{3} \sigma$ 
and get
$$\| |u|^2\|_{X_{\pm}^{\sigma -1,-\frac{1}{2}+}} \lesssim 
T^{1-\frac{\sigma}{2}} 
\|u\|^2_{X^{0,\frac{1}{2}-}} \, . $$
\end{itemize}
Thus we conclude that (\ref{12}),(\ref{13}) hold:
\begin{itemize}
\item
if $-\frac{1}{2} \le \sigma \le 0$ with $k=\frac{1}{2}+\frac{\sigma}{2}-$ , 
$l=1-$ $ \Rightarrow k+l = \frac{3}{2}+ \frac{\sigma}{2}- \ge \frac{5}{4}- \, 
,$
\item if $0 \le \sigma \le 1$  with $k=\frac{1}{2}+\frac{\sigma}{2}-$ , 
$l=1-\frac{\sigma}{2}-$ $ \Rightarrow k+l = \frac{3}{2}-  \, ,$
\item if $1 \le \sigma < \frac{3}{2} $  with $k=1-$ , $l=1-\frac{\sigma}{2}-$ $ 
\Rightarrow k+l = 2 -\frac{\sigma}{2}+ > \frac{5}{4} - \, .$ 
\end{itemize}
\end{proof}

\begin{proof} [Proof of Theorem \ref{Theorem 3.2}]
By persistence of higher regularity it suffices to consider the case $s=0$ and 
$-\frac{1}{2} \le \sigma < \frac{3}{2}$.
We first use our local well-posedness result Theorem \ref{Theorem 2.1} which 
gives under our assumptions a local solution. Because $\|u(t)\|_{L^2}$ is 
conserved this solution exists as long as $\|n_+(t)\|_{H^{\sigma}} + 
\|n_-(t)\|_{H^{\sigma}}$ remains bounded. If this is the case for any $t > 0$ 
we 
are done. Otherwise we can suppose that at some time $t$ we have
$$ \|n_+(t)\|_{H^{\sigma}} +  \|n_-(t)\|_{H^{\sigma}} \gg \|u(t)\|_{L^2}^2 +1 = 
\|u_0\|_{L^2}^2 +1 \, . $$
Take this time as initial time $t=0$ so that
$$ \|n_{+0}\|_{H^{\sigma}} +  \|n_{-0}\|_{H^{\sigma}} \gg \|u_0\|_{L^2}^2 +1 \, 
. $$
We want to apply now our modified local well-posedness result Proposition 
\ref{Prop. 3.1}. The estimates (\ref{12}) and (\ref{13}) are fulfilled by 
Proposition \ref{Prop. 3.2} with $k+l \ge \frac{5}{4}- > 1$ and $1 > k,l \ge 
\frac{1}{4}- $. Estimate (\ref{16}) is also fulfilled.  We now choose $T$ such 
that (\ref{14}) and (\ref{15'}) are satisfied, namely
$$ T \sim \min \left( \frac{1}{(\|n_{+0}\|_{H^{\sigma}} + 
\|n_{-0}\|_{H^{\sigma}})^{\frac{1}{l}}} , \frac{1}{\|u_0\|_{L^2}^{\frac{1}{k}}} 
, 1 \right) \, . $$
Then (\ref{15}) is automatically satisfied, because (\ref{14}) holds and $ 
\|n_{+0}\|_{H^{\sigma}} + \|n_{-0}\|_{H^{\sigma}} \gtrsim \|u_0\|_{L^2} $. 
Using 
(\ref{13'}) we see that it is possible to use this local existence theorem $m$ 
times with intervals of length $T$, before $\|n_+(t)\|_{H^{\sigma}} + 
\|n_-(t)\|_{H^{\sigma}}$ at most doubles. Here we have
$$m \sim \frac{\|n_{+0}\|_{H^{\sigma}} + \|n_{-0}\|_{H^{\sigma}}}{T^k 
\|u_0\|_{L^2}^2 } \, . $$
After $m$ iterations we arrive at the time
\begin{eqnarray*}
mT & \sim & \frac{ T^{1-k} (\|n_{+0}\|_{H^{\sigma}} + 
\|n_{-0}\|_{H^{\sigma}})}{\|u_0\|_{L^2}^2 }\\
& \sim & \frac{ \min ( \frac{ \|n_{+0}\|_{H^{\sigma}} + \|n_{-0}\|_{H^{\sigma}} 
}{ (\|n_{+0}\|_{H^{\sigma}} + \|n_{-0}\|_{H^{\sigma}})^{\frac{1-k}{l}}} , 
\frac{ 
\|n_{+0}\|_{H^{\sigma}} + \|n_{-0}\|_{H^{\sigma}} }{ 
\|u_0\|_{L^2}^{\frac{1-k}{k}}} , \|n_{+0}\|_{H^{\sigma}} + 
\|n_{-0}\|_{H^{\sigma}} ) }{\|u_0\|_{L^2}^2} \\
& \sim & \min ( \frac{ (\|n_{+0}\|_{H^{\sigma}} + 
\|n_{-0}\|_{H^{\sigma}})^{\frac{k+l-1}{l}} }{\|u_0\|_{L^2}^2} , \frac{ 
\|n_{+0}\|_{H^{\sigma}} + \|n_{-0}\|_{H^{\sigma}} }{ 
\|u_0\|_{L^2}^{\frac{1-k}{k}+2}} , \\
& & \quad \quad \frac{\|n_{+0}\|_{H^{\sigma}} + 
\|n_{-0}\|_{H^{\sigma}}}{\|u_0\|_{L^2}^2} )\\
& \gtrsim & \min ( \frac{1}{\|u_0\|_{L^2}^{3+}},1 )
\end{eqnarray*}
using $k+l > 1$ , $k \ge \frac{1}{4}-$ and $ \|n_{+0}\|_{H^{\sigma}} +  
\|n_{-0}\|_{H^{\sigma}} \gg \|u_0\|_{L^2}^2$ . This is independent of  $ 
\|n_{+0}\|_{H^{\sigma}} +  \|n_{-0}\|_{H^{\sigma}}$. Using conservation of 
$\|u(t)\|_{L^2}$ again it is thus possible to repeat the whole procedure with 
time steps of equal length. This proves the global existence result.

In the range $ -\frac{1}{2} \le \sigma < \frac{1}{4}$ we can give a much easier 
proof using Strichartz' estimates for the Klein-Gordon equation as follows. In 
order to estimate the wave part we get from the integral equation (\ref{9}):
$$ \|n_{\pm}\|_{L^{\infty}((0,T), H^{\sigma})} \lesssim \|n_{\pm 
0}\|_{H^{\sigma}} + \||u|^2\|_{L^{\frac{4}{3}-}((0,T), H^{\sigma + \rho,1+}) }\, 
$$
where we defined $\tilde{q}= 4+$ , $\tilde{r}=\infty -$ such that 
$\frac{2}{\tilde{q}}+\frac{1}{\tilde{r}} = 
\frac{1}{2}$, and moreover 
$0=1+\rho-2(\frac{1}{2}-\frac{1}{\tilde{r}})+\frac{1}{\tilde{q}}$ 
$\Longleftrightarrow$ $\rho = -\frac{1}{4}+$ , so that $\sigma + \rho < 0$ . 
Thus by Sobolev's embedding and conservation of $\|u(t)\|_{L^2}$:
$$ \|n_{\pm}\|_{L^{\infty}((0,T), H^{\sigma})} \lesssim \|n_{\pm 
0}\|_{H^{\sigma}} + \||u|^2\|_{L^{\frac{4}{3}-}((0,T), L^1)} \lesssim \|n_{\pm 
0}\|_{H^{\sigma}} + T^{\frac{3}{4}-} \|u_0\|_{L^2}^2 \, $$
which implies global existence.
\end{proof}

\begin{proof} [Proof of Theorem \ref{Theorem 5.1} for $D=2$]
Using persistence of regularity it suffices to consider the case $s=0$ , $0 \le 
\sigma \le 1$. Let $T \le 1$ and $ \frac{1}{q} + \frac{1}{r} = \frac{1}{2} $ , 
$r < \infty $.
Using the notation from the proof of Proposition \ref{Prop. 3.1} we get by 
Strichartz' estimates for the Schr\"odinger equation:
\begin{eqnarray*}
\|S_0 u\|_{L^{\infty}((0,T),L^2) \cap L^q((0,T), L^r)} & \lesssim & 
\|u_0\|_{L^2} + 
\|nu\|_{L^{\tilde{q}'}((0,T), L^{\tilde{r}'})} \\
& \lesssim & \|u_0\|_{L^2} + \|n\|_{L^{\bar{q}}((0,T),L^{\bar{r}})}  
\|u\|_{L^{q_0}((0,T), 
L^{r_0})}  \, .
\end{eqnarray*}
Here $\frac{1}{r_0}:= \frac{1}{2}-\epsilon$ and $\frac{1}{q_0} = \epsilon $ for 
a sufficiently small $\epsilon > 0$.\\
{\bf a.} $ 0 \le \sigma < 1 $.\\ 
Here $\frac{1}{\tilde{q}'} = \frac{1}{2} + \frac{\sigma}{2} + \epsilon$ ,
$\frac{1}{\tilde{q}} = \frac{1}{2} - \frac{\sigma}{2} - \epsilon$ , 
$ 
\frac{1}{\tilde{r}'} = 1-\frac{\sigma}{2}-\epsilon$, $ 
\frac{1}{\tilde{r}} = \frac{\sigma}{2}+\epsilon$,  so that $\frac{1}{\tilde{q}} 
+ \frac{1}{\tilde{r}} = \frac{1}{2}$. Choose
$\frac{1}{\bar{r}} = 
\frac{1}{2}-\frac{\sigma}{2}$, so that  $ H_x^{\sigma} \subset L_x^{\bar{r}}$ 
and $\frac{1}{\bar{q}} = 
\frac{1}{2}+ \frac{\sigma}{2}$. Then $\frac{1}{\tilde{r}'} = \frac{1}{\bar{r}} 
+ 
\frac{1}{r_0}$ and  $\frac{1}{\tilde{q}'} = \frac{1}{\bar{q}} + 
\frac{1}{q_0}$. Thus we get the estimate
\begin{eqnarray*}
\|S_0 u\|_{L^{\infty}((0,T),L^2) \cap L^q((0,T), L^r)}  
& \lesssim & \|u_0\|_{L^2} + T^{\frac{1}{2}+\frac{\sigma}{2}} 
\|n\|_{L^{\infty}((0,T),H^{\sigma})}  \|u\|_{L^{q_0}((0,T), L^{r_0})} \\
& \lesssim & \|u_0\|_{L^2} + T^{\frac{1}{2}} 
\|n\|_{L^{\infty}((0,T),H^{\sigma})}  \|u\|_{L^{q_0}((0,T), L^{r_0})}
 \, ,
\end{eqnarray*}
because $T\le 1$.\\
{\bf b.} $ \sigma = 1$. \\
We choose $\frac{1}{\bar{r}} = \epsilon$ , $\frac{1}{\tilde{r}'} = \frac{1}{2}$ 
, $\frac{1}{\tilde{q}} = 0$ , $\frac{1}{\tilde{q}'} = 1$ , $\frac{1}{\bar{q}} = 
1-\epsilon$, so that again $\frac{1}{\tilde{q}} + \frac{1}{\tilde{r}} = 
\frac{1}{2}$ ,  $\frac{1}{\tilde{r}'} = \frac{1}{\bar{r}} + 
\frac{1}{r_0}$ ,  $\frac{1}{\tilde{q}'} = \frac{1}{\bar{q}} + 
\frac{1}{q_0}$ and $ H^{\sigma} \subset L_x^{\bar{r}}$. We thus get the 
estimate
\begin{eqnarray*}
\|S_0 u\|_{L^{\infty}((0,T),L^2) \cap L^q((0,T), L^r)}  
& \lesssim & \|u_0\|_{L^2} + T^{1-\epsilon} 
\|n\|_{L^{\infty}((0,T),H^{\sigma})}  \|u\|_{L^{q_0}((0,T), L^{r_0})} \\
& \lesssim & \|u_0\|_{L^2} + T^{\frac{1}{2}} 
\|n\|_{L^{\infty}((0,T),H^{\sigma})}  \|u\|_{L^{q_0}((0,T), L^{r_0})}
 \, .
\end{eqnarray*}
Moreover
\begin{eqnarray*}
\|S_{\pm} n_{\pm}\|_{L^{\infty}((0,T), H^{\sigma})} & \le &\|n_{\pm 
0}\|_{H^{\sigma}} 
+ c \||u|^2\|_{L^1((0,T),H^{\sigma -1})} \\
& \le & \|n_{\pm 0}\|_{H^{\sigma}} + c \|u\|_{L^2((0,T), L^{2p})} \\
& \le & \|n_{\pm 0}\|_{H^{\sigma}} + c T^{1-\frac{\sigma}{2}} 
\|u\|_{L^{q_1}((0,T), L^{r_1})}
\\ & \le & \|n_{\pm 0}\|_{H^{\sigma}} + c T^{\frac{1}{2}} 
\|u\|_{L^{q_1}((0,T), L^{r_1})} \, ,
\end{eqnarray*}
where $\frac{1}{p}=\frac{1}{2}-\frac{\sigma -1}{2}$, so that $L_x^{p} \subset 
H_x^{\sigma -1}$, and $\frac{1}{r_1} = \frac{1}{2} - \frac{\sigma}{4}$, 
$\frac{1}{q_1} = \frac{\sigma}{4}$, so that $\frac{1}{q_1} + \frac{1}{r_1}  = 
\frac{1}{2}$.
Giving similar estimates for the differences $S_0 u - S_0 \tilde{u}$ and 
$S_{\pm} n_{\pm} - S_{\pm} \tilde{n}_{\pm}$ and choosing $T$ subject to the 
conditions
\begin{eqnarray}
\label{5.1'}
T^{\frac{1}{2}} (\|n_{+0}\|_{H^{\sigma}} + 
\|n_{-0}\|_{H^{\sigma}}) & \lesssim & 1 \\
\label{5.2'}
T^{\frac{1}{2}} \|u_0\|_{L^2} & \lesssim &1 \\
\label{5.4'}
T^{\frac{1}{2}} \|u_0\|_{L^2}^2 & \lesssim& 
\|n_{+0}\|_{H^{\sigma}} + \|n_{-0}\|_{H^{\sigma}} \, , 
\end{eqnarray}
then Banach's fixed point theorem shows that there exists a unique solution of 
our system of integral equations (\ref{8}),(\ref{9}) on $[0,T]$ such that
$$ \|u\|_{L^{\infty}((0,T), L^2) \cap L^q((0,T), L^r)} \lesssim \|u_0\|_{L^2} $$
 and
\begin{equation}
\label{5.6'}
\|n_{\pm}\|_{L^{\infty}((0,T), H^{\sigma})} \le \|n_{\pm 0}\|_{H^{\sigma}} + c 
T^{\frac{1}{2}} \|u_0\|_{L^2}^2 \, .
\end{equation}
Using conservation of mass we have $\|u(t)\|_{L^2} = \|u_0\|_{L^2}$, and thus 
get a global solution unless we have after a number of iterations
$$ \|n_+(t)\|_{H^{\sigma}} + \|n_-(t)\|_{H^{\sigma}} \gg \|u_0\|_{L^2}^2 + 1 \, 
, $$
which we thus may suppose. Take this time as initial time $t=0$ so that 
$$ \|n_{+0}\|_{H^{\sigma}} + \|n_{-0}\|_{H^{\sigma}} \gg \|u_0\|_{L^2}^2 + 1 \, 
. $$
Then (\ref{5.4'}) is automatically satisfied. Using (\ref{5.1'}) we choose
\begin{equation}
\label{5.7'}
T^{\frac{1}{2}} \sim \frac{1}{\|n_{+0}\|_{H^{\sigma}} + 
\|n_{-0}\|_{H^{\sigma}}} \, .
\end{equation}
Then (\ref{5.2'}) is also satisfied. By (\ref{5.6'}) we see that after $m$ 
iterations of size (\ref{5.7'}) the quantity $\|n_{+0}\|_{H^{\sigma}} + 
\|n_{-0}\|_{H^{\sigma}}$ at most doubles, where
$$ m \sim \frac{\|n_{+0}\|_{H^{\sigma}} + 
\|n_{-0}\|_{H^{\sigma}}}{T^{\frac{1}{2}} \|u_0\|_{L^2}^2} \, . 
$$
The total time after $m$ iterations is
$$ mT \sim T^{\frac{1}{2}} \frac{\|n_{+0}\|_{H^{\sigma}} + 
\|n_{-0}\|_{H^{\sigma}}}{\|u_0\|_{L^2}^2} \sim \frac{1}{\|u_0\|_{L^2}^2} \, , 
$$
by (\ref{5.7'}), which is independent of $\|n_{\pm 0}\|_{H^{\sigma}}$.
We can now repeat the whole procedure with time steps of equal length, thus 
leading to a global solution.
\end{proof}

\section{Global well-posedness results for the case $D=3$}
We generalize the argument of Colliander-Holmer-Tzirakis 
\cite{CHT} for data $u_0 \in H^s \, , \, n_0 \in H^{\sigma} \, , \, n_1 \in 
H^{\sigma -1}$ from the case $\sigma = s \ge 0$ to the region $ s \ge 0$ , $ s 
- \frac{1}{2} < \sigma 
\le s +1$. 

\begin{proof}[Proof of Theorem \ref{Theorem 5.1} for $D=3$]
Using persistence of regularity it again suffices to consider the case $s=0$ , 
$0 \le \sigma \le 1$. Let $T \le 1$ and $ \frac{2}{q} + \frac{3}{r} = 
\frac{3}{2} $ .
Similarly as in the 2D case we estimate
\begin{eqnarray*}
\|S_0 u\|_{L^{\infty}((0,T),L^2) \cap L^q((0,T), L^r)} & \lesssim & 
\|u_0\|_{L^2} + 
\|nu\|_{L^{\tilde{q}'}((0,T), L^{\tilde{r}'})} \\
& \lesssim & \|u_0\|_{L^2} + \|n\|_{L^{\bar{q}}((0,T),L^{\bar{r}})}  
\|u\|_{L^4((0,T), L^3)} \\
& \lesssim & \|u_0\|_{L^2} + T^{\frac{1}{4}+\frac{\sigma}{2}} 
\|n\|_{L^{\infty}((0,T),H^{\sigma})}  \|u\|_{L^4((0,T), L^3)} \, .
\end{eqnarray*}
Here $\frac{1}{\tilde{q}'} = \frac{1}{2} + \frac{\sigma}{2}$ , $ 
\frac{1}{\tilde{r}'} = \frac{5}{6}-\frac{\sigma}{3}$, $\frac{1}{\bar{r}} = 
\frac{1}{2}-\frac{\sigma}{3}$ ($\Rightarrow \frac{1}{\tilde{q}} = 
\frac{1}{2}-\frac{\sigma}{2} $ , $ \frac{1}{\tilde{r}} = \frac{1}{6} + 
\frac{\sigma}{3}$ and $H^{\sigma} \subset L^{\bar{r}}$), so that 
$\frac{2}{\tilde{q}} + 
\frac{3}{\tilde{r}} = \frac{3}{2}$ , thus Strichartz' estimate applies. 
Furthermore we define $\frac{1}{\bar{q}} = \frac{1}{4} + \frac{\sigma}{2} $ so 
that $\frac{1}{\tilde{q}'} = \frac{1}{\bar{q}} + \frac{1}{4}$ , and also 
$\frac{1}{\tilde{r}'} = \frac{1}{\bar{r}} + \frac{1}{3}$ so that H\"older's 
estimate applies. 
Moreover
\begin{eqnarray*}
\|S_{\pm} n_{\pm}\|_{L^{\infty}((0,T), H^{\sigma})} & \le &\|n_{\pm 
0}\|_{H^{\sigma}} 
+ c \||u|^2\|_{L^1((0,T), H^{\sigma -1})} \\
& \le & \|n_{\pm 0}\|_{H^{\sigma}} + c \|u\|^2_{L^2((0,T), L^{2p})} \\
& \le & \|n_{\pm 0}\|_{H^{\sigma}} + c T^{\frac{3}{4}-\frac{\sigma}{2}} 
\|u\|^2_{L^{q_0}((0,T), L^{r_0})} \, ,
\end{eqnarray*}
where $\frac{1}{p}=\frac{5}{6}-\frac{\sigma}{3}$, so that $L_x^{p} \subset 
H_x^{\sigma -1}$, and $\frac{1}{r_0} = \frac{1}{2p}$, $\frac{1}{q_0} = 
\frac{1}{8} + 
\frac{\sigma}{4}$, so that $\frac{2}{q_0} + \frac{3}{r_0} = 
\frac{1}{4}+\frac{\sigma}{2}+\frac{5}{4}-\frac{\sigma}{2} = \frac{3}{2}$.
Giving similar estimates for the differences $S_0 u - S_0 \tilde{u}$ and 
$S_{\pm} n_{\pm} - S_{\pm} \tilde{n}_{\pm}$ and choosing $T$ subject to the 
conditions
\begin{eqnarray}
\label{5.1}
T^{\frac{1}{4}+\frac{\sigma}{2}} (\|n_{+0}\|_{H^{\sigma}} + 
\|n_{-0}\|_{H^{\sigma}}) & \lesssim &1 \\
\label{5.2}
T^{\frac{1}{4}+\frac{\sigma}{2}} \|u_0\|_{L^2} & \lesssim &1 \\
\label{5.3}
T^{\frac{3}{4}-\frac{\sigma}{2}} \|u_0\|_{L^2} & \lesssim &1 \\
\label{5.4}
T^{\frac{3}{4}+\frac{\sigma}{2}} \|u_0\|_{L^2}^2 & \lesssim &
\|n_{+0}\|_{H^{\sigma}} + \|n_{-0}\|_{H^{\sigma}} \, , 
\end{eqnarray}
then Banach's fixed point theorem shows that there exists a unique solution of 
our system of integral equations (\ref{8}),(\ref{9}) on $[0,T]$ such that
$$ \|u\|_{L^{\infty}((0,T), L^2) \cap L^q((0,T), L^r)} \lesssim \|u_0\|_{L^2} $$
 and
\begin{equation}
\label{5.6}
\|n_{\pm}\|_{L^{\infty}((0,T), H^{\sigma})} \le \|n_{\pm 0}\|_{H^{\sigma}} + c 
T^{\frac{3}{4}-\frac{\sigma}{2}} \|u_0\|_{L^2}^2 \, .
\end{equation}
Using conservation of mass we have $\|u(t)\|_{L^2} = \|u_0\|_{L^2}$, and thus 
get a global solution unless we have after a number of iterations
$$ \|n_+(t)\|_{H^{\sigma}} + \|n_-(t)\|_{H^{\sigma}} \gg \|u_0\|_{L^2}^3 + 1 \, 
, $$
which we thus may suppose. Take this time as initial time $t=0$ so that 
$$ \|n_{+0}\|_{H^{\sigma}} + \|n_{-0}\|_{H^{\sigma}} \gg \|u_0\|_{L^2}^3 + 1 \, 
. $$
Then (\ref{5.4}) is automatically satisfied. Using (\ref{5.1}) we choose
\begin{equation}
\label{5.7}
T^{\frac{1}{4}+\frac{\sigma}{2}} \sim \frac{1}{\|n_{+0}\|_{H^{\sigma}} + 
\|n_{-0}\|_{H^{\sigma}}} \, .
\end{equation}
Then (\ref{5.2}) is also satisfied, because $\|u_0\|_{L^2} \ll 
\|n_{+0}\|_{H^{\sigma}} + \|n_{-0}\|_{H^{\sigma}} $ and
$$ (T^{\frac{3}{4}-\frac{\sigma}{2}} \|u_0\|_{L^2})^3 \le (T^{\frac{1}{4}} 
\|u_0\|_{L^2})^3 \ll T^{\frac{3}{4}} (\|n_{+0}\|_{H^{\sigma}} + 
\|n_{-0}\|_{H^{\sigma}}) \sim T^{\frac{3}{4}} 
T^{-(\frac{1}{4}+\frac{\sigma}{2})} \lesssim 1 \, , $$
so that (\ref{5.3}) is also satisfied. By (\ref{5.6}) we see that after $m$ 
iterations of size (\ref{5.7}) the quantity $\|n_{+0}\|_{H^{\sigma}} + 
\|n_{-0}\|_{H^{\sigma}}$ at most doubles, where
$$ m \sim \frac{\|n_{+0}\|_{H^{\sigma}} + 
\|n_{-0}\|_{H^{\sigma}}}{T^{\frac{3}{4}-\frac{\sigma}{2}} \|u_0\|_{L^2}^2} \, . 
$$
The total time after $m$ iterations is
$$ mT \sim T^{\frac{1}{4}+\frac{\sigma}{2}} \frac{\|n_{+0}\|_{H^{\sigma}} + 
\|n_{-0}\|_{H^{\sigma}}}{\|u_0\|_{L^2}^2} \sim \frac{1}{\|u_0\|_{L^2}^2} \, , 
$$
by (\ref{5.7}), which is independent of $\|n_{\pm 0}\|_{H^{\sigma}}$.
We can now repeat the whole procedure with time steps of equal length, thus 
leading to a global solution.
\end{proof}

\begin{proof}[Proof of Theorem \ref{Theorem 1.6'}]
Using persistence of regularity it suffices to consider the case $s=0$ , 
$-\frac{1}{2} < \sigma < 0$. Using the local wellposedness theorem 
\cite[Theorem 1.1]{P} and conservation of mass we only have to give a bound for 
$\|n(t)\|_{H^{\sigma}} + \|\partial_t n(t)\|_{H^{\sigma -1}}$. We use 
Strichartz' estimate for the Klein-Gordon equation and get
$$ \|n\|_{L^{\infty}((0,T),H^{\sigma})} \lesssim \|n_0\|_{H^{\sigma}} + 
\|n_1\|_{H^{\sigma-1}} + \| |u|^2 \|_{L^{\tilde{q}'}((0,T),H^{\sigma + 
\rho,\tilde{r}'})} \, , $$
where $\tilde{r} = \infty -$ , $\tilde{q} = 2 +  $ , so that 
$\frac{1}{\tilde{q}} + \frac{1}{\tilde{r}} = \frac{1}{2}$ . Moreover $0 = 
1+\rho -3(\frac{1}{2} - \frac{1}{\tilde{r}}) + \frac{1}{\tilde{q}}$ 
$\Longleftrightarrow$ $\rho = 0-$ . Thus Strichartz' estimate applies. By 
Sobolev's embedding we have $L^1({\mathbb R}^3) \subset H^{\rho + 
\sigma,\tilde{r}'}{(\mathbb R}^3) = H^{ \sigma -,1+}{(\mathbb R}^3)$ . Thus we 
arrive at
\begin{eqnarray*} 
 \|n\|_{L^{\infty}((0,T),H^{\sigma})} & \lesssim & \|n_0\|_{H^{\sigma}} + 
\|n_1\|_{H^{\sigma-1}} + \|u \|^2_{L^{4-}((0,T),L^2)} \\
& \lesssim & \|n_0\|_{H^{\sigma}} + \|n_1\|_{H^{\sigma-1}} + T^{\frac{1}{2}+} 
\|u_0 \|^2_{L^2}\, .
\end{eqnarray*}
Similarly $\|\partial_t n\|_{L^{\infty}((0,T),H^{\sigma -1})}$ can be 
estimated.  
\end{proof}

\end{document}